\documentclass{amsart}
\usepackage{mathtools}
\usepackage{enumitem}
\usepackage{lscape}

\usepackage[backend=biber,citestyle=alphabetic,bibstyle=alphabetic,maxbibnames=99]{biblatex}
\makeatletter
\renewcommand\section{\@startsection {section}{1}{\z@}%
                                   {3.5ex \@plus 1ex \@minus .2ex}%
                                   {-1em}%
                                   {\normalfont\large\bfseries}}
\makeatother

\numberwithin{equation}{section}
\renewbibmacro*{doi+eprint+url}{%
  \iftoggle{bbx:doi}
    {\iffieldundef{url}{\printfield{doi}}{}}
    {}%
  \newunit\newblock
  \iftoggle{bbx:eprint}
    {\usebibmacro{eprint}}
    {}%
  \newunit\newblock
  \iftoggle{bbx:url}
    {\usebibmacro{url+urldate}}
    {}}
    
\AtEveryBibitem{%
  \ifentrytype{book}{%
  \clearfield{pages}%
  \clearfield{url}%
  }{%
  }%
}

\newtheorem{thm}{Theorem}[section]
\newtheorem{cor}[thm]{Corollary}
\newtheorem{lem}[thm]{Lemma}
\newtheorem{prop}[thm]{Proposition}
\newtheorem{fact}[thm]{Fact}
{\theoremstyle{definition} \newtheorem{defn}[thm]{Definition} }
{}
{\theoremstyle{remark} \newtheorem{smallrem}[thm]{Remark}}
{\theoremstyle{remark} }

\newtheorem{bigclm}[thm]{Claim}
{\theoremstyle{remark} }

\newtheorem{keythm}{Theorem}

\DeclarePairedDelimiter{\bracepair}{\lbrace}{\rbrace}

\DeclarePairedDelimiter{\parenpair}{(}{)}
\DeclarePairedDelimiter{\vertpair}{\vert}{\vert}
\DeclarePairedDelimiter{\Vertpair}{\Vert}{\Vert}
\newcommand{\abs}[1]{\vertpair*{#1}}

\makeatletter
\newcommand{\raisemath}[1]{\mathpalette{\raisem@th{#1}}}
\newcommand{\raisem@th}[3]{\raisebox{#1}{$#2#3$}}
\makeatother

\newcommand{\declare}{\equiv}

\newcommand{\bigo}[1]{O \left( #1 \right)}

\newcommand{\bigoinv}[1]{O \left( \frac{1}{#1} \right)}

\newcommand{\card}[1]{\# \left( #1 \right)}

\renewcommand{\Re}{\operatorname{Re}}
\renewcommand{\Im}{\operatorname{Im}}

 \newcommand{\RR}{\mathbb{R}}
\newcommand{\CC}{\mathbb{C}} 
 
\newcommand{\NN}{\mathbb{N}} 
 
\newcommand{\Nz}{\mathbb{N}_{0}}

\newcommand{\gotP}{\mathcal{P}}
\newcommand{\gotQ}{\mathcal{Q}}
\newcommand{\gotR}{\mathcal{R}}
\newcommand{\gotS}{\mathcal{S}}
\newcommand{\gotT}{\mathcal{T}}
\newcommand{\gotU}{\mathcal{U}}

\newcommand{\leibder}[3]{\frac{d^{#2}}{{d#1}^{#2}} \left( #3 \right)}
\newcommand{\leibshort}[3]{\frac{d^{#2}{#3}}{{d#1}^{#2}}}

\newcommand{\leiboneat}[3]{\left. \frac{d {#2}}{{d#1}} \right\vert_{#1 = #3}}
\newcommand{\deloneat}[3]{\left. \frac{\partial}{\partial #1} \left( #2 \right) \right\vert_{#1 = #3}}

\DeclareMathOperator{\unbdddom}{\mathfrak{D}}
\newcommand{\opdom}[1]{\unbdddom \left( #1 \right)}
\DeclareMathOperator{\spectry}{Sp}
\newcommand{\spec}{\spectry}

\DeclareMathOperator{\wron}{Wr}
\newcommand{\wronsk}[3]{\wron \left[ #1, #2 \right] \left( #3 \right)}

\newcommand{\sobh}[2]{\mathcal{H}^{#2}\parenpair{#1}}
\newcommand{\sobw}[3]{\mathcal{W}^{#2, #3}\parenpair{#1}}

\newcommand{\ltwopair}[3]{\parenpair*{#1, #2}_{L^2(#3)}}
\newcommand{\ltwonorm}[2]{\Vertpair*{#1}_{L^2(#2)}}

\DeclareFontFamily{U}{mathx}{\hyphenchar\font45}
\DeclareFontShape{U}{mathx}{m}{n}{
      <5> <6> <7> <8> <9> <10>
      <10.95> <12> <14.4> <17.28> <20.74> <24.88>
      mathx10
      }{}
\DeclareSymbolFont{mathx}{U}{mathx}{m}{n}
\DeclareFontSubstitution{U}{mathx}{m}{n}
\DeclareMathAccent{\widecheck}{0}{mathx}{"71}
\DeclareMathAccent{\wideparen}{0}{mathx}{"75}

\newcommand{\pointmass}[1]{\delta \left( #1 \right)}

\newcommand{\honat}{L_{\text{HO}}^{0}}

\newcommand{\hopert}[2]{L_{\text{HO}}\parenpair*{#1, #2}}

\newcommand{\pcnpert}[2]{L_{\text{PC}}\parenpair*{#1, #2}}

\newcommand{\ourform}[2]{\mathfrak{t}_{#1, #2}}

\newcommand{\ourformlong}[4]{\mathfrak{t}_{#1, #2} \left( #3, #4 \right) }

\newcommand{\evensol}[2]{y_{\textnormal{even}}(#1; #2)}
\newcommand{\oddsol}[2]{y_{\textnormal{odd}}(#1; #2)}

\newcommand{\nuD}[1]{\mathsf{D}(#1)}
\newcommand{\nuP}[1]{\Phi(#1)}
\newcommand{\nuM}[1]{\mathsf{M}(#1)}

\makeatletter
\def\nuaseq#1{\expandafter\@nuaseq\csname c@#1\endcsname}
\def\@nuaseq#1{%
    \ifcase#1\or $\nu$\or $a$\else\@ctrerr\fi}
\def\anuseq#1{\expandafter\@anuseq\csname c@#1\endcsname}
\def\@anuseq#1{%
    \ifcase#1\or $a$\or $\nu$\else\@ctrerr\fi}
\def\eoseq#1{\expandafter\@nuaseq\csname c@#1\endcsname}
\def\@eoseq#1{%
    \ifcase#1\or e\or o\else\@ctrerr\fi}
\makeatother
{%
   \begin{subequations}
   
}
{%
    \end{subequations}\ignorespacesafterend
}
{%
   \begin{subequations}
   
}
{%
    \end{subequations}\ignorespacesafterend
}
{%
   \begin{subequations}
   
}
{%
    \end{subequations}\ignorespacesafterend
}
\newenvironment{romansub}%
{%
   \begin{subequations}
   
}
{%
    \end{subequations}\ignorespacesafterend
}

\title[Harmonic Oscillator Perturbation: Nonreal Eigenvalues]{Non-real eigenvalues of the Harmonic Oscillator perturbed by an odd, two-point \texorpdfstring{$\delta$}{delta}-potential}
\author{Charles Baker and Boris Mityagin}
\date{\today}

\begin{document}

\begin{abstract}
In this paper, we consider the perturbations of the Harmonic Oscillator Operator by an odd pair of point interactions: $z (\pointmass{x - b} -  \pointmass{x + b})$.  We study the spectrum by analyzing a convenient formula for the eigenvalue.  We conclude that if $z = ir$, $r$ real, as $r \to \infty$, the number of non-real eigenvalues tends to infinity.
\end{abstract}

\maketitle
\everymath{\displaystyle}

\section{Introduction}  
\subsection{Perturbations of the Harmonic Oscillator} We consider the \emph{harmonic oscillator operator},

\begin{subequations} \label{eq:honatdef}
\begin{align}
\honat y(x) &= - y^{\prime \prime}(x) + x^2 y(x), \label{eq:honatEQN}\\
\opdom{\honat} & \declare \bracepair{y(x) \in  \sobh{\RR}{2}: x^2 y(x) \in L^2(\RR)}.
\end{align}
\end{subequations} 
($\sobh{\RR}{k} = \sobw{\RR}{k}{2}$ is a Sobolev space, and we allow $y^{\prime \prime}$ to be a distributional derivative.)  The operator is of compact resolvent, and the spectrum is the positive odd integers:
\[
\spec(\honat) = 2 \Nz + 1, \quad \Nz  \declare \NN \cup \bracepair{0}.
\]
(For an operator $T$,  $\spec(T)$ denotes its spectrum.)  The eigenfunction with eigenvalue $2n + 1$, $n \in \Nz$, is the $n$th \emph{Hermite function},
\begin{equation} \label{eq:hermitefcndef}
h_n(x) \declare \frac{1}{\sqrt{2^n n! \sqrt{\pi}}} e^{-x^2/2} H_n(x),
\end{equation}
where 
\begin{equation} \label{eq:hermitepolydef}
H_n(x) \declare (-1)^n e^{x^2} \leibder{x}{n}{e^{-x^2}}.
\end{equation}
is the $n$th \textit{Hermite Polynomial}.  We now consider a perturbation of the harmonic oscillator operator,
\begin{equation}
\hopert{z}{b} y(x) = \honat y(x) + z (\pointmass{x - b} - \pointmass{x + b}) y(x),  \quad b  > 0, z \in \CC.\label{eq:hopertEQN}
\end{equation}
The operator's construction is discussed in \cite{MiSiArt} (following \cite[Chapter VI]{Kato}), where it is noted that the operator has compact resolvent \cite[p. 5]{MiSiArt}.  In particular, the operator is defined by the quadratic-form method, for the quadratic form $\mathfrak{t}^{\textnormal{HO}}_{z, b}$ with domain $\opdom{\mathfrak{t}^{\textnormal{HO}}_{z, b}}  = \bracepair{f(x) \in \sobh{\RR}{1}: x f(x) \in L^2(\RR)}$ by
\[
\begin{split}
\mathfrak{t}^{\textnormal{HO}}_{z, b}(f(x), g(x)) & = \ltwopair{f^{\prime}(x)}{g^{\prime}(x)}{\RR} + \ltwopair{x f(x)}{x g(x)}{\RR} \\
& \quad + z \left( f(b) \overline{g(b)} -  f(-b) \overline{g(-b)} \right).
\end{split}
\]
In particular, for $f \in \opdom{\hopert{z}{b}} \subset \opdom{\mathfrak{t}^{\textnormal{HO}}_{z, b}}$ and $g \in \opdom{\mathfrak{t}^{\textnormal{HO}}_{z, b}}$,
\[
\ltwopair{\hopert{z}{b} f(x)}{g(x)}{\RR} = \mathfrak{t}^{\textnormal{HO}}_{z, b} (f(x), g(x)).
\]
This perturbation was studied in \cite{MiSiArt},  \cite{HCWArt},  \cite{MitPub2015}, and \cite{MitPub2016}, and a similar operator was studied in \cite{Demi05b}.  In particular, \cite{HCWArt} and \cite{Demi05b} gave numerical evidence that when $z = ir$, non-real eigenvalues exist for large enough $r$, and \cite{MitPub2015} proved that for $z = ir$, the number of non-real eigenvalues was bounded above by $M ( \abs{r} \log (e \abs{r}))^2$.  But how does the number of non-real eigenvalues change as $r$ increases?  We show --- and this is a main result of our paper --- that $\mathcal{N}_{\textnormal{HO}}(r) = \card{ \spec(\hopert{ir}{b}) \setminus \RR }$ increases to $\infty$.  
\begin{keythm} \label{thm:infinityHO}
For any fixed $b > 0$, 
\begin{equation} \label{eq:TheTheorem} 
\lim_{r \to \infty} \mathcal{N}_\textnormal{HO}(r) = \infty.
\end{equation}
\end{keythm}

Our approach is based on finding a function $M(\nu, b)$ (meromorphic in $\nu$, holomorphic in $b$) such that $\nu \in \CC \setminus \Nz$ is an eigenvalue of (the transformation of) $\hopert{z}{b}$~\ref{sec:eigenCond} for the derivations),
\begin{equation} \label{eq:maineqnIntro}
\begin{split}
M(\nu; b) &= \frac{1}{z^2}, \quad \text{ where }\\
M(\nu; b) &\declare \sqrt{\frac{2}{\pi}}  \Gamma(-\nu) D_{\nu}^2(b) \evensol{\nu}{b} \oddsol{\nu}{b}.
\end{split}
\end{equation}  
The zeros of $M(\nu, b)$ \emph{in the parameter} $\nu$, in particular of its factor $D_{\nu}(b)$, become important to find and analyze the asymptotics of the eigenvalues for $\abs{z}$ large.  We observe that there are infinitely many zeros of $\nuD{\nu} = D_{\nu}(b)$ (see Section~\ref{sec:zeroCount}), where we streamline some arguments with the work of F. W. Olver (\cite{Olver59b}, \cite{Olver61}), with tighter results on the growth rates of $\nuD{\nu}$ than necessary.  Around each zero $\lambda$ of $\nuD{\nu}$, for large $\abs{z}$, we find solutions of \eqref{eq:maineqnIntro} in a small neighborhood of $\lambda$, non-real for $z$ imaginary (see Sections~\ref{sec:nuNotInt} and \ref{sec:nuInt}); this is an important step in the completion of the proof (see Section~\ref{sec:Unify}).  
\subsection{Change of Variables} \label{subsect:changevar} We make a change of variables, since the \emph{Weber differential equation}, written in the form
\begin{equation} \label{eq:WeberEqnIntro}
- \leibshort{x}{2}{y} + \left( \frac{x^2}{4} - \frac{1}{2} \right) y(x) = \nu y(x), \, \, x \in \CC, \, \, \nu \in \CC.
\end{equation}
has its general solution far more studied than the the notation above,
\[
- \leibshort{x}{2}{y} + x^2 y(x) = \lambda y(x).
\] 
We define (with $z \in \CC$, $b > 0$)
\begin{equation} \label{eq:pcnpertdef}
\pcnpert{z}{b} y(x) = - y^{\prime \prime}(x) + \left( \frac{x^2}{4} - \frac{1}{2} \right) y(x) + z \left[ \pointmass{x - b} - \pointmass{x + b} \right] y(x).
\end{equation} 
The corresponding quadratic form, $\ourform{z}{b}$, has the same domain, $\opdom{\ourform{z}{b}} \declare \bracepair{f(x) \in \sobh{\RR}{1}: x f(x) \in L^2(\RR)}$, and is defined by
\begin{equation} \label{eq:ourformdefOne}
\begin{split}
\ourformlong{z}{b}{f}{g} &\declare \ltwopair{f^{\prime}(x)}{g^{\prime}(x)}{\RR} + \frac{1}{4}\ltwopair{xf(x)}{xg(x)}{\RR} - \frac{1}{2}\ltwopair{f(x)}{g(x)}{\RR} \\
& \quad + z \left( f(b) \overline{g(b)} - z f(-b) \overline{g(-b)} \right)\\
\end{split} 
\end{equation}
and again, for $f \in \opdom{\pcnpert{z}{b}} \subset \opdom{\ourform{z}{b}}$ and $g \in \opdom{\ourform{z}{b}}$,
\[
\ltwopair{\pcnpert{z}{b} f(x)}{g(x)}{\RR} =  \ourformlong{z}{b}{f}{g}.
\]

One may check that if $S x = x \sqrt{2}$ is the dilation on the real line, and $Tf = f \circ S$ is the corresponding operator on $L^2(\RR)$, we have that 
\begin{equation} \label{eq:finalconvert}
\pcnpert{z}{b} = \frac{1}{2} T^{-1} \circ \hopert{z \sqrt{2}}{\frac{b}{\sqrt{2}}} \circ T - \frac{1}{2} I,
\end{equation}
and
\begin{equation} \label{eq:finalspec}
\spec (\pcnpert{z}{b}) = \frac{\spec \hopert{z \sqrt{2}}{\frac{b}{\sqrt{2}}} - 1}{2}.
\end{equation}
Hence, defining $\mathcal{N}_{\textnormal{PC}}(r) =\card{ \spec(\pcnpert{ir}{b}) \setminus \RR }$, \eqref{eq:TheTheorem} is equivalent to the claim that $\lim_{r \to \infty} \mathcal{N}_{\textnormal{PC}}(r) = \infty$.  

\section{Reciprocal Gamma Function} \label{sec:prelimGam}
Following complex-analysis convention (e.g., \cite[p. 27]{LevinMainOriginal}), we define the \emph{entire} function
\[
\begin{split}
\frac{1}{\Gamma(\zeta)} &= e^{\gamma\zeta} \prod_{n = 1}^{\infty} \left( 1 + \frac{\zeta}{n} \right) e^{\zeta/n}\\
\gamma & = \lim_{n \to\infty} \left[ \left( 1 + \frac{1}{2} + \dotsb + \frac{1}{n} \right) - \log(n + 1) \right].
\end{split}
\]
and its multiplicative inverse is the usual Gamma function, a meromorphic function with poles at the nonpositive integers.  In particular, for $n$ a nonnegative integer, $\frac{1}{\Gamma(-n)} = 0$.  The \emph{Stirling approximation} for the gamma function yields the estimate \cite[Chap.1, Sec. 11, p. 27]{LevinMainOriginal}, with $\abs{z} = r$, and outside of circles of fixed width about the points in $- \Nz$,
\begin{equation} \label{eq:Stir}
\ln \Gamma(z) = \left( z - \frac{1}{2}\right) \ln z  - z + \frac{1}{2} \ln (2 \pi) + \bigoinv{r},
\end{equation}
therefore
\[
\ln \left( M_{1/\Gamma}(r)\right)  : = \ln \left( \sup_{\abs{z} = r} \abs{ \frac{1}{\Gamma(z)}} \right)\sim r \log r, \quad r \to \infty.
\]
(In this text, $f(x) \sim g(x)$ as $x \to +\infty$ means $\lim_{x \to \infty} \frac{f(x)}{g(x)} = 1$.)
Some properties of the Gamma function on the positive real line are as follows.  
\begin{fact} \label{fact:GamIncr}
On $(0, \infty)$, $\log \Gamma(x)$ is positive and convex, and $\Gamma(1) = \Gamma(2) = 1$; hence, for $x$ real and positive, $\Gamma(x)$ has a unique minimum in $[1, 2]$.  In particular, $\Gamma(x)$ is increasing on $[2, \infty)$.
\end{fact}
See , e.g., \cite[Section 7.7, p. 179]{Conway} and \cite[Section 2.2, Exercise 2.4, p. 39]{OlverMain}. \vspace{2 ex} 

\begin{fact} \label{fact:doubler} For $2w \not\in -\Nz$,
\begin{equation} \label{eq:doubler}
\Gamma(2 w) = \frac{2^{2w - 1}}{\sqrt{\pi}} \Gamma(w) \Gamma\left(w + \frac{1}{2} \right)
\end{equation}
\end{fact}
See, e.g., \cite[Chapter 2, Section 1, (1.08), p. 35]{OlverMain}. \vspace{2 ex}

\begin{fact}  \label{fact:GamInt} If $\Re \zeta > 0$, $\mu > 0$, and $\Re z > 0$, 
\begin{equation} \label{eq:keybound}
\int_0^{\infty} \exp (- z t^{\mu}) t^{\zeta - 1} \, dt = \frac{1}{\mu} \Gamma \left( \frac{\zeta}{\mu} \right) \frac{1}{z^{\zeta/\mu}},
\end{equation}
where the implicit logarithms use the principal branch of the logarithm.
\end{fact}
See, e.g., \cite[Chapter 2, Section 1, Exer. 1.1., p. 38]{OlverMain}.

\section{The Weber Differential Equation and its Solutions} \label{sec:prelimWeber} \subsection{Notation}  The Weber differential equation can be written in \emph{either} of the forms 
\begin{romansub}
\begin{align} 
- \leibshort{x}{2}{y} + \left( \frac{x^2}{4} - \frac{1}{2} \right) y(x) &= \nu y(x), \quad x \in \CC, \quad \nu \in \CC, \label{eq:ourWeber}\\
- \leibshort{x}{2}{y} + \left( \frac{1}{4} x^2 + a \right) y(x) &= 0, \quad x \in \CC, \quad a \in \CC. \label{eq:theirWeber}
\end{align}\label{eq:introWeber}
\end{romansub}
These notations are equivalent under the rule
\begin{equation}
a = - \nu - \frac{1}{2} \label{eq:anuequiv}
\end{equation}
and \eqref{eq:anuequiv} is assumed throughout the rest of the paper.  We will use \eqref{eq:ourWeber}, as it is the choice of coordinates used by the relevant references \cite{Demi05b} and \cite{MOS}, and because of a clearer connection to the harmonic oscillator.  We mention \eqref{eq:theirWeber} because of the frequent use of this form in the literature (e.g., \cite[Section 6.6]{OlverMain},  \cite{DLMFParab}, and \cite{Dean66}).  

The solutions of \eqref{eq:introWeber} are called \emph{parabolic cylinder functions}.  We now discuss some particular solutions.
\subsection{Solutions decaying as \texorpdfstring{$x \to \infty$}{x to infinity}: \texorpdfstring{$D_{\nu}(x)$ or $U(a, x)$}{D sub nu (x) or U(a, x)}}  \label{subsect:ourSoln} One solution of \eqref{eq:ourWeber}, denoted $D_{\nu}(x)$, is a solution of \eqref{eq:ourWeber} that decays as $x \to + \infty$; more precisely (e.g., \cite[Section 12.9(i), (12.9.1)]{DLMFParab})
\[
D_{\nu}(x) \sim x^{\nu} e^{-x^2/4}, \quad x \to + \infty.
\]
$D_{\nu}(x)$ may also be characterized by the values,
\begin{equation} \label{eq:dnvalues}
D_{\nu}(0) = \frac{2^{ \nu/2} \sqrt{\pi}}{\Gamma \left(- \frac{\nu}{2} + \frac{1}{2}\right)}, \quad \deloneat{x}{D_{\nu}(x)}{0} = - \frac{2^{(\nu + 1)/2 }\sqrt{\pi}}{\Gamma \left( - \frac{\nu}{2} \right)} ,
\end{equation}
the connection being derivable from the integral formulas (e.g., \cite[Section 8.1.4, p. 328]{MOS}),
\begin{equation} \label{eq:dnuintegrals}
D_{\nu}(x) = \begin{cases} \frac{ e^{ - x^2/4}}{\Gamma(-\nu)} \int_0^{\infty} t^{-\nu - 1} e^{-t^2/2  - xt} \, dt, & \Re \nu < 0,\\
\sqrt{\frac{2}{\pi}} e^{x^2/4} \int_0^{\infty} e^{-t^2/2} \cos \left( \frac{\pi \nu}{2} - xt \right) t^{\nu} \, dt, & \Re \nu > -1. \end{cases}
\end{equation}
As the coefficients of the differential equation \eqref{eq:ourWeber} are jointly continuous in $x$ and $\nu$, and holomorphic in each variable separately, and since the initial conditions are homomorphic in $\nu$, a standard continuation-of-parameters result (e.g., \cite[Section 5.3, Thm. 3.2, p. 146]{OlverMain}) ensures that for each $x$, $D_{\nu}(x)$ is holomorphic in $\nu$.

If the Weber parabolic cylinder equation is written in $a$-notation, i.e. \eqref{eq:theirWeber}, then the function named $D_{\nu}(x)$ in $\nu$-notation is denoted $U(a, x)$. 

In the sequel, we will abuse language and call $D_{\nu}(x) = U(a, x)$ ``the'' parabolic cylinder function.  

\subsection{Transformations of the parabolic cylinder function}  Certain transformations of the parabolic cylinder function still satisfy the Weber differential equation \eqref{eq:ourWeber} (see, e.g.,\cite[Section 8.1.1, p. 324, and Section 8.1.3, p. 327]{MOS}).  
\begin{fact} \label{fact:sols} For the differential equation
\begin{equation} \label{eq:WeberLoc}
- \leibshort{x}{2}{y} + \left( \frac{x^2}{4} - \frac{1}{2} \right) y(x) = \nu y(x), \, \, x \in \CC, \, \, \nu \in \CC,
\end{equation}
solutions include $D_{\nu}(x)$, $D_{\nu}(-x)$, $D_{- \nu - 1}(ix)$, and $D_{- \nu - 1}(-ix)$.  Some Wronskians include:
\begin{subequations} \label{eq:wrons}
\begin{align}
\wron [D_{\nu}(x), D_{\nu}(-x)] &= \frac{\sqrt{\pi}}{\Gamma(- \nu)}, \label{eq:pmwron} \\
\wron [D_{\nu}(x), D_{- \nu - 1}(\pm ix)] & = \exp \left( \mp i \frac{\pi}{2} (\nu + 1) \right) \label{eq:rotatewron}.
\end{align}
\end{subequations}
\end{fact}

\begin{smallrem} \label{rem:constWronsk}
Since \eqref{eq:WeberLoc} has no $\frac{dy}{dx}$ term, the Wronskians of any two of its solutions are constant functions.  
\end{smallrem}

\subsection{The even and odd solutions}  \label{subsect:evenOdd} We now present standard \emph{even} and \emph{odd} solutions to \eqref{eq:ourWeber}, given by (e.g., \cite[Chap. 5, Exercise 3.5, pp. 147--148]{OlverMain})
\begin{romansub} 
\begin{align} 
\evensol{\nu}{x} & \declare e^{-x^2/4} \left[ 1 + \left( -\nu \right) \frac{x^2}{2!} + \left( - \nu \right)\left( - \nu + 2\right) \frac{x^4}{4!} + \dotsb \right] \label{eq:yEvenDirect} \\
\intertext{and}
\oddsol{\nu}{x} & \declare e^{-x^2/4} \left[ x + \left( - \nu  + 1 \right) \frac{x^3}{3!} + \left( - \nu  + 1\right)\left( - \nu + 3\right) \frac{x^5}{5!} + \dotsb \right]. \label{eq:yOddDirect}
\end{align} \label{eq:yParityDirect}
\end{romansub}
For future reference, we note that from \eqref{eq:yEvenDirect} and \eqref{eq:yOddDirect}, one sees that
\begin{romansub} \label{eq:ZeroIC}
\begin{align}
\evensol{\nu}{0} & = 1, && \deloneat{x}{\evensol{\nu}{x}}{0} = 0, \label{eq:yEvenZeroIC}\\
\oddsol{\nu}{0} & = 0, && \deloneat{x}{\oddsol{\nu}{x}}{0} = 1, \label{eq:yOddZeroIC}
\end{align}
\end{romansub}
so these new solutions are also holomorphic in $x$ and $\nu$.  

By \eqref{eq:ZeroIC} and \eqref{eq:dnvalues}, we may write $D_{\nu}(x)$ in terms of the even and odd solutions (e.g., \cite[Section 12.4]{DLMFParab}):
\begin{equation} \label{eq:dnudecomp}
\begin{split}
D_{\nu}(x) & =  D_{\nu}(0) \evensol{\nu}{x} +  \deloneat{x}{D_{\nu}(x)}{0} \oddsol{\nu}{x}\\
& = \frac{2^{ \nu/2} \sqrt{\pi}}{\Gamma \left(- \frac{\nu}{2} + \frac{1}{2}\right)} \evensol{\nu}{x} - \frac{2^{(\nu + 1)/2 }\sqrt{\pi}}{\Gamma \left( - \frac{\nu}{2} \right)} \oddsol{\nu}{x}.
\end{split}
\end{equation} 

\subsection{(Non)interference of the zeros of different solutions}

In the sequel, we  fix $x$ and discuss the zeroes of $D_{\nu}(x)$ \emph{in the parameter} $\nu$, and argue to what extent the zeroes \emph{in the parameter} of $\evensol{\nu}{x}$ and $\oddsol{\nu}{x}$ do or do not interfere.

\begin{lem} \label{lem:zerointerference} Fix $x_0 \in \CC$.\\
\begin{enumerate}[label = (\alph*).]  
\item If $\nu \not\in \Nz$, and $D_{\nu}(x_0) = 0$, then $\evensol{\nu}{x_0} \neq 0$ and $\oddsol{\nu}{x_0} \neq 0$.\\
\item if $n \in \Nz$, and $D_{n}(x_0) = 0$, then exactly one of $\evensol{n}{x_0}$, $\oddsol{n}{x_0}$ is $0$.
\end{enumerate}
\end{lem}

\begin{proof}[Proof, Part (a)]  Fix $\nu, x_0 \in \CC$, $\nu \not\in \Nz$, such that $D_{\nu}(x_0) = 0$.  If $\evensol{\nu}{x_0} = 0$, then by \eqref{eq:dnudecomp},
\[
\begin{split}
\oddsol{\nu}{x_0} &= \frac{\Gamma \left( - \frac{\nu}{2} \right)}{2^{(\nu + 1)/2 }\sqrt{\pi}} \left(  \frac{2^{ \nu/2} \sqrt{\pi}}{\Gamma \left(- \frac{\nu}{2} + \frac{1}{2}\right)} \evensol{\nu}{x_0} - D_{\nu}(x_0) \right)\\
& = \frac{\Gamma \left( - \frac{\nu}{2} \right)}{2^{(\nu + 1)/2 }\sqrt{\pi}} \left(  \frac{2^{ \nu/2} \sqrt{\pi}}{\Gamma \left(- \frac{\nu}{2} + \frac{1}{2}\right)} \cdot  0 -0 \right) = 0. 
\end{split} 
\]
Hence, the Wronskian of $\evensol{\nu}{x}$ and $\oddsol{\nu}{x}$ would be $0$ at $x = x_0$.  Yet by \eqref{eq:ZeroIC},
\[
\wronsk{y_{\text{even}}}{y_{\text{odd}}}{0} = 1 \cdot 1 - 0 \cdot 0 = 1,
\]
and hence the Wronskian cannot vanish anywhere in $\CC$; in particular, $\evensol{\nu}{x_0}$ and $\oddsol{\nu}{x_0}$ can never be simultaneously 0.  Contradiction.  Similarly if $\oddsol{\nu}{x_0} = 0$.
\end{proof}
\begin{proof}[Proof, Part (b)]  Note that if $n = 2k$ is a nonnegative \emph{even} integer, then by \eqref{eq:dnudecomp},
\[
\begin{split}
D_{2k}(x) &= \frac{2^{k} \sqrt{\pi}}{\Gamma \left(- k + \frac{1}{2}\right)} \evensol{2k}{x} - \frac{2^{k + (1/2) }\sqrt{\pi}}{\Gamma \left( - k\right)} \oddsol{2k}{x} \\
& = \frac{2^{k} \sqrt{\pi}}{\Gamma \left(- k + \frac{1}{2}\right)} \evensol{2k}{x} - 0,
\end{split}
\]
so $D_{2k}(x)$ is a nonzero multiple of $\evensol{2k}{x}$, and these functions have the same zeroes.  The proof works similarly if $n = 2k + 1$ is a positive odd integer.  
\end{proof}

\section{Eigenvalue conditions}  \label{sec:eigenCond}
\subsection{\texorpdfstring{$L^2(\RR)$}{L2(R)} solutions}  We return to finding the eigenvalues of 
\begin{equation} \label{eq:pcnpertrepeat}
\pcnpert{z}{b} y(x) = - y^{\prime \prime}(x) + \left( \frac{x^2}{4} - \frac{1}{2} \right) y(x) + z \left[ \pointmass{x - b} - \pointmass{x + b} \right] y(x), \quad z \in \CC, b > 0
\end{equation}
\begin{fact}[Folklore] \label{fact:jumpconds} Fix $b > 0$, $z \in \CC$, and $y(x) \in \sobh{\RR}{1}$.  Then the conditions
\begin{enumerate}[label = (\roman*)]
\item $y(x) \in \opdom{\pcnpert{z}{b}}$, and 
\item $y(x)$ is an eigenfunction of \eqref{eq:pcnpertrepeat} with eigenvalue $\nu$, i.e.,
\[
\pcnpert{z}{b} y(x) = \nu y(x),
\]
\end{enumerate}
hold if and only if $y(x)$ is a $C^{\infty}$ solution of \eqref{eq:WeberLoc} on $(b,\infty)$, $(-b, b)$, and $(-\infty, -b)$, and we have the following jumps in the first derivative
\begin{equation} \label{eq:jumpconds}
\begin{split}
y^{\prime}(-b+) - y^{\prime}(-b-) & = - z y(-b),\\
y^{\prime}(b+) - y^{\prime}(b-) & = z y(b).
\end{split}
\end{equation}
\end{fact}
(For $c$ real, $f(x)$ a function on $\RR$, $f(c \pm) = \lim_{x \to c^{\pm}} f(x)$.)  Combining Fact~\ref{fact:sols} and Fact~\ref{fact:jumpconds}, we have the following.
\begin{lem} \label{lem:types}
Fix $b > 0$ and $z \in \CC$.  If $y(x)$ is an eigenfunction of $\pcnpert{z}{b}$ with eigenvalue $\nu$, then 
\begin{equation} \label{eq:form}
y(x) = \begin{cases} \beta D_{\nu}(-x), & x \leq - b,\\
\sigma \evensol{\nu}{x} + \tau \oddsol{\nu}{x}, & -b < x < b,\\
\alpha D_{\nu}(x), & x \geq b,  \end{cases}
\end{equation}
for some complex $\alpha, \beta, \sigma, \tau$.  
\end{lem}
\begin{proof}
We choose a basis of solutions to \eqref{eq:WeberLoc} on each subinterval.
\begin{description} 
\item[{On  $[b, \infty)$}]  $\bracepair{D_{\nu}(x), D_{- \nu - 1}(ix)}$ is a basis because their Wronskian is nonzero (see \eqref{eq:rotatewron}).
\item[On $(-b, b)$]  $\bracepair{\evensol{\nu}{x}, \oddsol{\nu}{x}}$ is a basis here.  Note that by \eqref{eq:ZeroIC}, 
\[
\wronsk{\evensol{\nu}{x}}{\oddsol{\nu}{x}}{0} = 1 \cdot 1 - 0 \cdot 0 = 1,
\]
and since $\evensol{\nu}{x}$ and $\oddsol{\nu}{x}$ satisfy the same differential equation, their Wronskian is therefore never $0$, so we indeed have a basis.
\item[{On $(-\infty, -b]$}]  We choose $\bracepair{D_{\nu}(-x), D_{- \nu - 1}(ix)}$ is a basis.  The Wronskian is nonzero, because by the Chain Rule,
\[
\wronsk{f(cx)}{g(cx)}{t} = c \wronsk{f(x)}{g(x)}{ct},
\]
so that again by \eqref{eq:rotatewron},
\[
\wron [D_{\nu}(-x), D_{- \nu - 1}(ix)] = - \wron [D_{\nu}(x), D_{- \nu - 1}(-ix) ] = -  \exp \left( i \frac{\pi}{2} (\nu + 1) \right) \neq 0.
\]
\end{description}
Eigenfunctions are solutions to the unperturbed differential equation \eqref{eq:WeberLoc} on each subinterval, by Fact~\ref{fact:jumpconds}; therefore 
\begin{equation} \label{eq:semiform}
y(x) = \begin{cases} \beta D_{\nu}(-x) + t D_{- \nu - 1}(ix), & x \leq - b,\\
\sigma \evensol{\nu}{x} + \tau \oddsol{\nu}{x}, & -b < x < b,\\
\alpha D_{\nu}(x) + s D_{- \nu - 1}(ix), & x \geq b.  \end{cases}
\end{equation}
for some complex constants $\beta, t, \sigma, \tau, \alpha, s$.  Yet we have the known asymptotic as $\abs{x} \to \infty$, (\cite[Section 8.1.6, p. 331]{MOS}),
\[
D_{\nu}(x) = e^{-x^2/4} e^{\nu \log x} \left( 1 - \frac{\nu(\nu - 1)}{2x^2} + \bigo{\abs{x}^{-4}} \right), \quad \abs{\arg x} < \frac{3\pi}{4}.
\]
Applying this to $D_{- \nu - 1}(\pm ix)$, if $\abs{\arg x} < \frac{\pi}{4}$,
\[
D_{- \nu - 1}(\pm ix) = e^{x^2/4} e^{( - \nu - 1) [(\log x) \pm  i\pi/2]} \left( 1 + \frac{( - \nu - 1)(-\nu - 2)}{2x^2} + \bigo{\abs{x}^{-4}} \right).
\]
Hence, in \eqref{eq:semiform}, line 3, $s = 0$, otherwise the solutions would be growing in magnitude as $x \to +\infty$, which is incompatible with $y \in L^2(\RR)$.  In the same way, we can show in \eqref{eq:semiform}, line 1, $t = 0$, if we rewrite $D_{- \nu - 1}(ix) = D_{- \nu - 1}(- i \cdot (-x) )$, so that as $x \to -\infty$, $\arg (-x) = 0$ and we may apply the above asymptotic.    
\end{proof}
\subsection{Boundary conditions at \texorpdfstring{$\pm b$}{plus-or-minus b}}
Suppose that $y(x)$ is indeed an eigenfunction of $\pcnpert{z}{b}$ with eigenvalue $\nu$.  Then by Lemma~\ref{lem:types}, we have that 
\begin{equation} \label{eq:formrepeat}
y(x) = \begin{cases} \beta D_{\nu}(-x), & x \leq - b,\\
\sigma \evensol{\nu}{x} + \tau \oddsol{\nu}{x}, & -b < x <b,\\
\alpha D_{\nu}(x), & x \geq b.  \end{cases}
\end{equation}
Yet functions in the domain of $\pcnpert{z}{b}$ are continuous, so we must have
\begin{equation}
\begin{split}
y(b+) & = y(b-)\\
\alpha D_{\nu}(b) & = \sigma \evensol{\nu}{b} + \tau \oddsol{\nu}{b}
\end{split}
\end{equation}
and
\begin{equation}
\begin{split}
y(-b+) & = y(-b-)\\
\sigma \evensol{\nu}{-b} + \tau \oddsol{\nu}{-b} & = \beta D_{\nu}(- (-b))\\
\sigma \evensol{\nu}{b} - \tau \oddsol{\nu}{b} & = \beta D_{\nu}(b) .
\end{split}
\end{equation}
Similarly, the jump condition at $+b$ (i.e., \eqref{eq:jumpconds} becomes
\begin{equation}
\begin{split}
y^{\prime}(b+) - y^{\prime}(b-) & = z y(b)\\
\alpha \deloneat{x}{D_{\nu}(x)}{b} - (\sigma \deloneat{x}{\evensol{\nu}{x}}{b} + \tau \deloneat{x}{\oddsol{\nu}{x}}{b}) & = z \alpha D_{\nu}(b)\\
%\alpha \left(z D_{\nu}(b) -  \deloneat{x}{D_{\nu}(x)}{b} \right) +  \sigma \deloneat{x}{\evensol{\nu}{x}}{b} + \tau \deloneat{x}{\oddsol{\nu}{x}}{b} & = 0.
\end{split}
\end{equation}
The jump condition at $-b$ becomes (by function parity and the Chain Rule)
\begin{equation}
\begin{split}
y^{\prime}(-b+) - y^{\prime}(-b-) & = -z y(b)\\
\left(\sigma \deloneat{x}{\evensol{\nu}{x}}{-b} + \tau \deloneat{x}{\oddsol{\nu}{x}}{-b} \right) - \beta \deloneat{x}{D_{\nu}(-x)}{-b} & = -z \beta D_{\nu}(-(-b))\\
- \sigma \deloneat{x}{\evensol{\nu}{x}}{b} + \tau \deloneat{x}{\oddsol{\nu}{x}}{b} + \beta \deloneat{x}{D_{\nu}(x)}{b} & = -z \beta D_{\nu}(b)\\
%\beta\left (z D_{\nu}(b) + \deloneat{x}{D_{\nu}(x)}{b} \right) - \sigma \deloneat{x}{\evensol{\nu}{x}}{b} + \tau \deloneat{x}{\oddsol{\nu}{x}}{b}  & = 0.  
\end{split}
\end{equation}
Putting this all together, and letting
\begin{subequations} \label{eq:gothdeclareone}
\begin{align}
\gotP &\declare D_{\nu}(b), & \gotQ &\declare \deloneat{x}{D_{\nu}(x)}{b}, \\
\gotR &\declare \evensol{\nu}{b},& \gotS &\declare \deloneat{x}{\evensol{\nu}{x}}{b},\\
\gotT & \declare \oddsol{\nu}{b}, &\gotU & \declare \deloneat{x}{\oddsol{\nu}{x}}{b},
\end{align}
\end{subequations}
we have that 
\begin{equation}
\begin{pmatrix}
-\gotP            & 0 & \gotR & \gotT\\
0                &- \gotP   & \gotR &  -\gotT \\
 z \gotP - \gotQ & 0 & \gotS &  \gotU \\
0 & z \gotP  + \gotQ & - \gotS &  \gotU
\end{pmatrix}
\cdot \begin{pmatrix}
\alpha \\ \beta \\ \sigma \\ \tau
\end{pmatrix} = \begin{pmatrix}
0 \\ 0 \\ 0 \\ 0
\end{pmatrix} \, .
\end{equation}
This has nontrivial solutions (i.e., the $\nu$-eigenspace is nontrivial) if and only if the determinant is nonzero, i.e., if and only if
\begin{equation} \label{eq:det}
2 (\gotR \gotQ - \gotS \gotP)(\gotT \gotQ - \gotU \gotP) -2 z^2 \gotP^2 \gotR \gotT= 0.
\end{equation}

Yet recalling the definitions of $\gotP$, etc. (i.e., \eqref{eq:gothdeclareone}), we see that 
\[
\gotR\gotQ - \gotS\gotP  = \wronsk{\evensol{\nu}{x}}{D_{\nu}(x)}{b}.
\]
By the decomposition of $D_{\nu}(x)$ into the even and odd terms, however, this becomes
\[
\begin{split}
& \wronsk{\evensol{\nu}{x}}{D_{\nu}(x)}{b} \\
= &\wronsk{\evensol{\nu}{x}}{\frac{2^{ \nu/2} \sqrt{\pi}}{\Gamma \left(- \frac{\nu}{2} + \frac{1}{2}\right)} \evensol{\nu}{x} - \frac{2^{(\nu + 1)/2 }\sqrt{\pi}}{\Gamma \left( - \frac{\nu}{2} \right)} \oddsol{\nu}{x}}{b}\\
= & - \frac{2^{(\nu + 1)/2 }\sqrt{\pi}}{\Gamma \left( - \frac{\nu}{2} \right)} \wronsk{\evensol{\nu}{x}}{\oddsol{\nu}{x}}{b}.
\end{split}
\]
As the Wronskians of solutions to \eqref{eq:ourWeber} are constant functions (see Remark~\ref{rem:constWronsk}), zero can be chosen as the evaluation point, and we get 
\[
\begin{split}
\wronsk{\evensol{\nu}{x}}{D_{\nu}(x)}{b} &=   - \frac{2^{(\nu + 1)/2 }\sqrt{\pi}}{\Gamma \left( - \frac{\nu}{2} \right)} \wronsk{\evensol{\nu}{x}}{\oddsol{\nu}{x}}{0}\\
& = - \frac{2^{(\nu + 1)/2 }\sqrt{\pi}}{\Gamma \left( - \frac{\nu}{2} \right)} .
\end{split}
\]
Similarly, one has
\[
\begin{split}
\gotT \gotQ - \gotU \gotP &= \wronsk{\oddsol{\nu}{x}}{D_{\nu}(x)}{b}\\
& = \frac{2^{ \nu/2} \sqrt{\pi}}{\Gamma \left(- \frac{\nu}{2} + \frac{1}{2}\right)}\wronsk{\oddsol{\nu}{x}}{\evensol{\nu}{x}}{0}\\
& = - \frac{2^{ \nu/2} \sqrt{\pi}}{\Gamma \left(- \frac{\nu}{2} + \frac{1}{2}\right)}.
\end{split}
\]
Altogether, then, \eqref{eq:det} becomes
\[
2 \left( \frac{2^{\nu + (1/2)} \pi}{\Gamma( - \frac{\nu}{2}) \Gamma \left(- \frac{\nu}{2} + \frac{1}{2}\right)}  - z^2 D_{\nu}^2(b) \evensol{\nu}{b} \oddsol{\nu}{b} \right) = 0
\]
By using the Gamma-function double-angle formula (e.g., Fact~\ref{fact:doubler}, \eqref{eq:doubler}),
\begin{equation} \label{eq:doublerRepeat}
\Gamma(2 w) = \frac{2^{2w - 1}}{\sqrt{\pi}} \Gamma(w) \Gamma\left(w + \frac{1}{2} \right), \quad 2w \not \in - \Nz,
\end{equation}
and applying with $w = - \frac{\nu}{2}$, $\nu \not\in \Nz$, we have
\begin{equation} \label{eq:flipdoubler}
\frac{2^{\nu + (1/2)} \pi}{\Gamma( - \frac{\nu}{2}) \Gamma \left(- \frac{\nu}{2} + \frac{1}{2}\right)} = \frac{\sqrt{\pi}}{\sqrt{2} \Gamma(-\nu)}.
\end{equation}
Yet  \eqref{eq:flipdoubler} holds for $\nu \in \Nz$ as well.  If $\nu = 2k$ is nonnegative and even, then both $\frac{1}{\Gamma(-\nu)} = \frac{1}{\Gamma(-2k)}$ and $\frac{1}{\Gamma( - \frac{\nu}{2})} = \frac{1}{\Gamma(-k)}$ are $0$, and if $\nu = 2k +1 $ is positive and odd, then both $\frac{1}{\Gamma(-\nu)} = \frac{1}{\Gamma(-2k - 1)}$ and $\frac{1}{\Gamma \left(- \frac{\nu}{2} + \frac{1}{2}\right)} = \frac{1}{-k}$ are $0$.  Therefore, we have:
\begin{prop} \label{prop:eigencondbasic} Fix $z \in \CC$ and $b \in (0, \infty)$.  Then $\nu \in \spec{\pcnpert{z}{b}}$ if and only if
\begin{equation} \label{eq:condbasic}
\frac{\sqrt{\pi}}{\sqrt{2}\Gamma(- \nu)}- z^2 D_{\nu}^2(b) \evensol{\nu}{b} \oddsol{\nu}{b}  = 0.
\end{equation}
\end{prop}  
\begin{smallrem}
This result is a slight generalization of \cite[Eq. 8, p. 1083]{Demi05b}.  To see the compatibility, take \eqref{eq:condbasic},  divide by $\frac{\sqrt{\pi}}{\sqrt{2}\Gamma(- \nu)}$, let $z = iV$, rename $b$ as $z_1$ and $\nu$ as $\xi$, and use Fact~\ref{fact:doubler} to rewrite the product of the odd and even solutions as a power of $2$, some Gamma functions, and $D_{\xi}^2(-z_1) - D_{\xi}^2(z_1)$.  
\end{smallrem}
\subsection{Separation of Variables}
We separate the variables, at the cost of making some functions in the equation meromorphic in $\nu$.  
\begin{cor}  \label{cor:Demiwithb}
Fix $z \in \CC$ and $b \in (0, \infty)$.  For $\nu \in \CC \setminus \Nz$,  $\nu \in \spec{\pcnpert{z}{b}}$ if and only if
\begin{equation} \label{eq:maineqnflipWithB}
M(\nu; b)= \frac{1}{z^2}, \end{equation}
where 
\begin{equation} \label{eq:maineqnelabWithB}
M(\nu; b) \declare \sqrt{\frac{2}{\pi}}  \Gamma(-\nu) D_{\nu}^2(b) \evensol{\nu}{b} \oddsol{\nu}{b}.
\end{equation}
The pole of $M(\nu; b)$ at $\nu = 0$ is not removable, so the function $M(\cdot; b)$ is not constant in $\nu$.  Also, $M(\nu; b)$ is real if $\nu$ is real.  
\end{cor}
\begin{proof}
For $\nu \not\in \Nz$, transforming \eqref{eq:condbasic} to \eqref{eq:maineqnflipWithB} is elementary algebra.  To demonstrate that the pole of $M(\nu, b)$ at $\nu = 0$ is not removable, we need to prove that $D_0(b)$, $\evensol{0}{b}$, and $\oddsol{0}{b}$ are not $0$.  
\begin{description}
\item[$D_{0}(b) \neq 0$] It is known (e.g, \cite[Section 8.1.2, p. 326]{MOS}) that $D_0(b) = \exp \left( - \frac{b^2}{4} \right)$, and this is nonzero.
\item[$\evensol{0}{b} \neq 0$] By \eqref{eq:yEvenDirect}, 
\[
\evensol{0}{b} = e^{-b^2/4} \left[ 1 + 0 \frac{b^2}{2!} + (0)(2) \frac{b^4}{4!} + \dotsb \right] = e^{-b^2/4},
\]
and this is nonzero.
\item[$\oddsol{0}{b} \neq 0$] By $\eqref{eq:yOddDirect}$ and $b > 0$,
\[
\begin{split}
\oddsol{0}{b} &= e^{-b^2/4} \left[ b + 1 \frac{b^3}{3!} + 1(3) \frac{b^5}{5!} + \dotsb \right] > 0.
\end{split}
\]
\end{description}
Since the pole at $\nu = 0$ is not removable, we must have that $\lim_{\nu \to 0} \abs{M(\nu; b)} = \infty$, and hence $M(\nu; b)$ is non-constant in $\nu$.  

The reality of $\evensol{\nu}{b}$ and $\oddsol{\nu}{b}$ if $\nu$ is real comes from the power-series expansions \eqref{eq:yParityDirect}, as all summands are real.  Then the decomposition of $D_{\nu}(x)$ in terms of the even and odd solutions, i.e. \eqref{eq:dnudecomp}, and the reality of the gamma function for real inputs in its domain, ensures that $D_{\nu}(b)$ is real for real inputs.  
\end{proof}

\subsection{Notation}  So far, $x$ has been the key variable, as our functions were functions of $x$, and all derivatives in \eqref{eq:ourWeber} are in $x$.  Now that the eigenvalue equation has been formed, the emphasis shifts, and $\nu$ becomes the primary variable in the sequel.   We choose a fixed $b > 0$ and suppress explicit references to $b$:
\begin{equation} \label{eq:removeb}
\begin{split}
\nuD{\nu} & \declare D_{\nu}(b),\\
\nuP{\nu} & \declare \sqrt{\frac{2}{\pi}}  \evensol{\nu}{b} \oddsol{\nu}{b},\\
\nuM{\nu} & \declare M(\nu; b).
\end{split}
\end{equation}
Thus, we rewrite Proposition~\ref{prop:eigencondbasic} and Corollary~\ref{cor:Demiwithb} as follows.
\begin{cor} \label{cor:Deminob}
Fix $z \in \CC$. $\nu \in \CC$ is in $\spec{\pcnpert{z}{b}}$ if and only if 
\begin{equation} \label{eq:condrelabel}
\frac{1}{\Gamma(- \nu)} - z^2 \nuD{\nu} \nuP{\nu} = 0.
\end{equation}
If, in addition, $\nu \in \CC \setminus \Nz$, then $\nu \in \spec{\pcnpert{z}{b}}$ if and only if
\begin{equation} \label{eq:maineqnflip}
\nuM{\nu} =  \frac{1}{z^2}, 
\end{equation}
where 
\begin{equation} \label{eq:maineqnelab}
\nuM{\nu} =  \Gamma(-\nu) [\nuD{\nu}]^2 \nuP{\nu}.
\end{equation}
The pole of $\mathsf{M}$ at $0$ is never removable, so the function $\nuM{\nu}$ is not constant.  Also, $\mathsf{M}$ is real-valued for $\nu$ real.  
\end{cor} 

\section{Zeros of Parabolic Cylinder Functions in the Parameter}  \label{sec:zeroCount}
An observation from \eqref{eq:maineqnflip} is that as $\abs{z} \to \infty$, $\abs{\frac{1}{z^2}} \to 0$, so the zeroes of $\nuM{\nu}$ outside $\Nz$, in particular of its factor $\mathsf{D}(\nu)$ (see \eqref{eq:maineqnelab}), are helpful in discerning the asymptotic behavior of the eigenvalues as $\abs{z}$ grows large.  For numerical confirmation of this idea, see Figure~\ref{fig:zimhints}. 

To ensure that we have zeroes of $\nuD{\nu}$ to work with, we prove:
\begin{prop} \label{prop:zerocount}
$\nuD{\nu}$ has infinitely many distinct zeroes.
\end{prop}
The statement is implied by Figure 2 (p. 280) of Dean's paper \cite{Dean66}, but we prefer to give a proof. Our approach uses the theory of entire functions, in particular the concept of (exponential) order of an entire function.

\begin{defn}
Let $f(z)$ be an entire function.  Then $f$ is of \emph{finite (exponential) order} if there exists $0 < \rho < \infty$ such that for all $r \geq R = R(\rho)$, 
\[
\max_{\abs{z} \leq r} \abs{f(z)} \leq \exp (r^{\rho}).
\]
If such a $\rho$ exists, we call the infimum of such $\rho$ the \emph{exponential order} (or simply \emph{order}) of $f$.

\end{defn}
\begin{landscape}
\begin{figure}
\centerline{\includegraphics[scale=0.8]{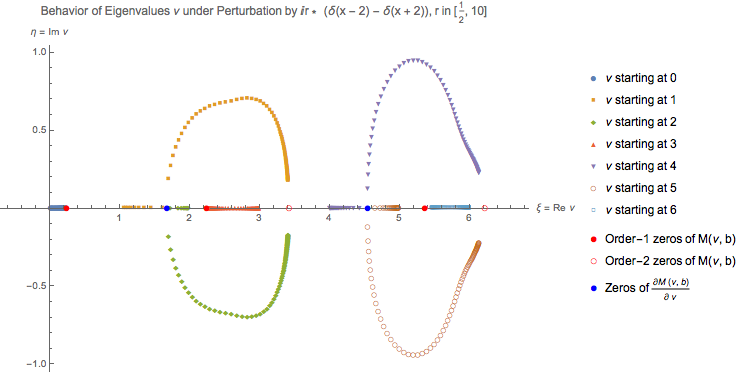}}
\caption{Some small eigenvalues of $\pcnpert{ir}{2}$, $r \in \bracepair*{\frac{1}{2} + \frac{j}{10}: 0 \leq j \leq 95}$.  Zeroes of $M(\nu; 2)$ and its first $\nu$-derivative are marked.  Each zero of $\nuD{\nu}$ eventually has a small neighborhood containing as many eigenvalues as the order of the zero.}
\label{fig:zimhints}
\end{figure}
\end{landscape}
For an entire function $f(\nu)$, define the maximum function $M_f(r)$, $r > 0$, by 
\begin{equation} \label{eq:maxfcndef} 
M_f(r) \declare \sup_{\abs{\nu} = r} \abs{f(\nu)};
\end{equation} 
by the Maximum Modulus Principle, $M_f(r) = \sup_{\abs{\nu} \leq r} \abs{f(\nu)}$.

\begin{smallrem}  It makes little difference to our entire-functions arguments whether the suppressed argument $b$ is real or complex; therefore, \emph{for this section only}, we will let $\beta$ be an arbitrary complex number and redefine $\nuD{\nu} = D_{\nu}(\beta)$; in this more general setting, the function is still at most order-$1$, maximal-type in $\nu$.
\end{smallrem}

For the moment, we assuming the following fact about the growth rate $\nuD{\nu}$, to be proven later. 
\begin{lem} \label{lem:ub}
$\nuD{\nu}$ is of exponential order at most $1$ in $\nu$, though possibly of maximal type; more specifically, we have the estimate, 
\begin{equation}
\log \bracepair*{M_{\mathsf{D}}(r)} \leq \frac{1}{4} r \log r + r \left(2 + \frac{\pi}{4} \right) + O(\log r).
\end{equation}
\end{lem}

We also require the fact that $\nuD{\nu}$ decays as $\nu \to - \infty$.
\begin{fact}[e.g., {\cite[Chapter 8, Section 8.1.6, p. 332]{MOS}}] \label{fact:decay} 
If $\displaystyle \abs{\arg(- \nu)} \leq \frac{\pi}{2}$,
\begin{equation} \label{eq:leftasympIntro}
\nuD{\nu} = \frac{1}{\sqrt{2}} \exp \left[ \frac{\nu}{2} \log(-\nu) - \frac{\nu}{2} - \beta \sqrt{-\nu} \right] \left\lbrace 1 + O (\abs{\nu} ^{-1/2}) \right\rbrace.
\end{equation}
In particular, for $\nu = - \xi$, $\xi > 0$, we have as $\xi \to + \infty$ that 
\begin{equation} \label{eq:leftasympOnLine}
\ln \abs{\nuD{- \xi}} = - \frac{\xi}{2} \log (\xi) + \frac{\xi}{2} - (\Re \beta) \sqrt{\xi}  - \frac{1}{2} \log(2) + \log \left\lbrace 1 + O (\xi^{-1/2}) \right\rbrace.
\end{equation}
\end{fact}

\begin{proof}[Proof of Proposition~\ref{prop:zerocount}, given Lemma~\ref{lem:ub}]  Suppose, by way of contradiction, that $\nuD{\nu}$ has only finitely many zeros.  Since $\nuD{\nu}$ is an entire function by the Weierstrass Factorization Theorem (e.g., \cite[Chapter 1, Section 3, Theorem 3, p. 8]{LevinMainOriginal}), we must have
\begin{equation} \label{eq:simpleformbegin}
\nuD{\nu} = P(\nu) e^{g(\nu)}, \quad P(\nu) \text{ a polynomial, }, g(\nu) \text{ entire.}
\end{equation}
Yet $\nuD{\nu}$ is a function of order at most $1$ by Lemma~\ref{lem:ub}, so $g(\nu)$ must be a degree-$1$ polynomial; else, it would not be order $1$.  Therefore,
\begin{equation} \label{eq:simpleform}
\nuD{\nu} = P(\nu) e^{c\nu + d}, \quad P(\nu) \text{ a polynomial, } c, d \in \CC.
\end{equation}
Let $M$ be the maximum modulus of the zeroes of $P(\nu)$, i.e., of $\nuD{\nu}$.  Then for $\nu = - \xi$, $\xi > M$, we must have $\nuD{- \xi} \neq 0$ and hence 
\[
\ln \abs{\nuD{-\xi}} = \ln \abs{P(-\xi)} - \xi \Re(c) + \Re (d) = (- \Re (c)) \xi + O_{\beta}(\ln \xi), \quad \xi \to + \infty.
\]
Yet by Fact~\ref{fact:decay}, \eqref{eq:leftasympOnLine}, $\ln \abs{\nuD{- \xi}} = - \frac{\xi}{2} \log (\xi) + \frac{\xi}{2} - O_{\beta}(\sqrt{\xi})$ as $\xi \to + \infty$.  Contradiction.  

Therefore, $\nuD{\nu}$ has infinitely many zeros.  
\end{proof}

\begin{proof}[Proof of Lemma~\ref{lem:ub}] We have (see \cite[Section 3, Theorem, (3.1), p. 813--4]{Olver61} ) that if $\nu = -a - \frac{1}{2} = \mu^2 - \frac{1}{2}$, $\beta = \mu t \sqrt{2}$, and with $\theta = \arg \mu$, $\abs{\theta} \leq \frac{\pi}{2}$,\footnote{This range of $\theta$ suffices, as $\abs{ \arg \mu} \leq \frac{\pi}{2}$ implies $\abs{\arg (\mu^2)} \leq \pi$, and $\nu = \mu^2 + \frac{1}{2}$; hence, all (sufficiently large) $\nu$ may be achieved by rewriting in this way.}
\begin{equation} \label{eq:ubbasic}
\abs{D_{\mu^2 - \frac{1}{2}} (\mu t \sqrt{2})} \leq k \frac{ \abs{(2e)^{- \, \frac{1}{4} \mu^2} \mu^{\frac{1}{2} \mu^2 }\exp (- \mu^2 \xi_{\theta}(t)) }}{1 + \abs{\mu}^{1/6} + \abs{\mu}^{\frac{1}{2}} \abs{t^2 - 1}^{1/4}}, 
\end{equation}
where $k$ is an unspecified constant, the powers $(2e)^{- \, \frac{1}{4} \mu^2}$ and $ \mu^{\frac{1}{2} \mu^2 }$ have their principal values, and $\xi_\theta(t)$ is a particular branch of the multivalued function
\[
\xi(t) = \int_1^t(u^2 - 1)^{1/2} \, du = \frac{1}{2} t (t^2 - 1)^{1/2} - \frac{1}{2} \ln \bracepair{t + \sqrt{t^2 - 1}},
\]
which is defined in the following way.  As we are interested in the growth rate as $\abs{\nu} \to \infty$, hence $\abs{\mu} = \sqrt{\abs{\nu + \frac{1}{2}}} \to \infty$, and $\beta$ is fixed, we will have $t = \frac{x}{\mu \sqrt{2}} \to 0$.  For $\abs{t} \leq \frac{1}{4}$,  the correct branches for our purposes are (\cite[Section 5, (5.8 -- 5.9), p. 143]{Olver59b}, \cite[Section 2, (2.8 -- 2.10), p. 813]{Olver61}), 
\[ 
\begin{split}
\xi_{\theta}(t) &= \pm i \eta(t),\\
\eta(t) & =  \int_t^1 (1 - u^2)^{1/2} \, du = \frac{1}{2} \left( \arccos(t) - t \sqrt{1 - t^2} \right),\\
\arccos(t) & = \int_t^1 \frac{du}{\sqrt{ 1 - u^2}}.
\end{split}  
\]
More specifically, the requirements are that (\cite[Section 2, (2.9 -- 2.10), p. 813]{Olver61}) 
\[
\left\lbrace \begin{aligned}
\xi_{\theta}(0) & = - \, \frac{i \pi}{4}, && - \, \frac{\pi}{2} \leq \theta < 0,\\ 
\xi_{0}(0 + 0i) & = - \, \frac{i \pi}{4}, \\
\xi_{0}(0 - 0i) & = \frac{i \pi}{4}, \\
\xi_{\theta}(0) & = \frac{i \pi}{4}, \quad && 0 < \theta \leq \frac{\pi}{2}. 
\end{aligned} \right.
\]
Since $\eta(0) = \frac{\pi}{4}$, it follows that for $\theta = 0$, $\xi_0(t) = i \eta(t)$ for $\Im t < 0$ and $t$ small, and $\xi_0(t) = - i \eta(t)$ for $\Im t > 0$ and $t$ small.  As $\theta$ changes, the branch cut moves, but wherever it is, below the branch cut, we use $\xi_{\theta}(t) = i \eta(t)$, and above it we use $\xi_{\theta}(t) = - i \eta(t)$, and on it, we will use whichever branch gives the larger upper bound for $\abs{\exp ( - \mu^2 \xi_{\theta}(t) )}$.   

To ensure that $\abs{t} \leq \frac{1}{4}$ above, it suffices to have  $\abs{\nu} \geq 8 \abs{\beta}^2 + \frac{1}{2}$, since $\mu^2 = \nu + \frac{1}{2}$, or $\abs{\mu^2} \geq \abs{\nu} - \frac{1}{2} \geq 8 \abs{\beta}^2$, or $\abs{\mu} \geq 2 \sqrt{2} \abs{\beta}$, and then
\[
\abs{t} = \frac{\abs{\beta}}{\mu \sqrt{2}} \leq \frac{\abs{\beta}}{4 \abs{\beta}} \leq \frac{1}{4}.
\]  

With $\abs{t} \leq \frac{1}{4}$, we have that $\frac{\sqrt{15}}{4} \leq \abs{ \sqrt{1 - t^2}} \leq 1$ and 
\[
\begin{split}
\abs{ \arccos t} & \leq \abs{\arccos(t) - \arccos(0)} + \abs{\arccos(0)}\\
& = \abs{ \int_t^0 \frac{du}{\sqrt{1 - u^2}}} + \frac{\pi}{2}\\
& \leq \frac{\pi}{2} + \abs{t} \cdot \frac{4}{\sqrt{15}} \leq \frac{\pi}{2} + \frac{1}{\sqrt{15}}. 
\end{split}
\]
so that (recalling that $\mu^2 = \nu + \frac{1}{2}$)
\[
\begin{split}
\abs{-  \mu^2 \xi_{\theta}(t) } &\leq \abs{\mu^2} \abs{ \pm  i} \abs{ \eta(t)}\\
& \leq  \abs{\nu + \frac{1}{2}}  \abs{ \frac{1}{2} } \left( \abs{\arccos t} + \abs{t} \abs{ \sqrt{1 - t^2}} \right)\\
& \leq \abs{\nu + \frac{1}{2}} \cdot \frac{1}{2} \cdot \left( \frac{\pi}{2} + \frac{1}{\sqrt{15}} + \frac{1}{4} \right)\\
& \leq \abs{\nu + \frac{1}{2}} \cdot \frac{3}{2}.
\end{split}
\]
Thus, 
\begin{equation} \label{eq:xibound}
\abs{\exp ( - \mu^2 \xi_{\theta}(t))} \leq \exp \abs{- \mu^2 \xi_\theta(t)} \leq \exp \left( \frac{3}{2} \abs{\nu} + \frac{3}{4} \right).
\end{equation}   

To bound the other terms in the numerator of \eqref{eq:ubbasic}, we have
\begin{equation} \label{eq:expbound}
\abs{(2e)^{- \frac{1}{4} \mu^2}} \leq (e)^{2 \cdot \abs{ - \frac{1}{4}} \cdot \abs{ \nu + \frac{1}{2} }} \leq \exp \left( \frac{1}{2} \abs{\nu } + \frac{1}{4} \right).  
\end{equation}
and
\begin{equation} \label{eq:mupowbound}
\begin{split}
\abs{ \mu^{\frac{1}{2} \mu^2}} & = \abs{ \exp \bracepair*{\frac{1}{2} \mu^2 \log \mu}}\\
& = \exp \bracepair*{ \frac{1}{2} \abs{\nu + \frac{1}{2}} \left( \log \abs{\mu} + \frac{\pi}{2} \right)} \\
& \sim \exp \left( \frac{1}{4} \abs{\nu + \frac{1}{2} } \log \abs{\nu + \frac{1}{2}} + \frac{\pi}{4} \abs{\nu + \frac{1}{2}} \right) , \, \, \, \abs{\nu} \to \infty.
\end{split} 
\end{equation}
Therefore, combining \eqref{eq:ubbasic}, \eqref{eq:xibound}, \eqref{eq:expbound}, and \eqref{eq:mupowbound},
\[
\begin{split}
\log M_{\mathsf{D}}(r) & \leq \log k + \frac{1}{2} r  + \frac{1}{4} r \log r + \frac{\pi}{4} r + \frac{3}{2} r +  O(\log r)\\
& = \frac{1}{4} r \log r + \left(2  + \frac{\pi}{4} \right) r + O(\log r)
\end{split}
\]
\end{proof} 

\section{Eigenvalue Localization Around Non-integer Roots of \texorpdfstring{$\mathsf{D}$}{D}} \label{sec:nuNotInt}
\subsection{Preliminaries}  
We now reformulate the the first part of Lemma ~\ref{lem:zerointerference} as
\begin{lem} \label{lem:commonzeroesinN}
If $\nu \not\in \Nz$ and $\nuD{\nu}  = 0$ , then $\nuP{\nu} \neq 0$.  
\end{lem}
As $\nuM{\nu} = \Gamma(-\nu) [\nuD{\nu}]^2 \nuP{\nu}$ (see \eqref{eq:maineqnelab}), zeroes of $\nuD{\nu}$ outside $\NN$ result in zeroes of exactly double the order for $\nuM{\nu}$, and hence should be close to solutions of $\nuM{\nu} = \frac{1}{z^2}$ for $\abs{z}$ large.  We compress the complex analysis into the statements below.  Let $\lambda \in \CC \setminus \NN$ be a zero of order exactly $m$ of $\mathsf{D}$.  Then $\lambda$ is a zero of order exactly $2m$ of $\mathsf{M}$; write this function locally as 
\begin{equation} \label{eq:localexpand}
\nuM{\nu} = c_{2m} (\nu - \lambda )^{2m} (1 + g(\nu - \lambda )), \quad c_{2m} \neq 0,
\end{equation}
where $g(0) = 0$ and $g$ is analytic in some disk; fix $R = R(\lambda) > 0$ such that $\abs{g(\zeta)} \leq \frac{1}{4}$ whenever $\abs{\zeta} \leq R$.   Let $B$ be the maximum of $\abs{g^{\prime}(\zeta)}$ on this neighborhood.  

\begin{lem} \label{lem:approachlem}  Fix $\rho > 0$.  There exists a constant $\delta = \delta(\lambda, \rho)$ such that if $\abs{z} > \frac{1}{\delta}$, there exist $2m$ solutions of $\nuM{\nu} = \frac{1}{z^2}$ (i.e., \eqref{eq:maineqnflip}), within the radius-$\rho$-neighborhood of $\lambda$; more precisely, if $\epsilon \in \CC$ is defined such that $\epsilon^{2m} = \frac{1}{z^2 c_{2m}}$, then the $2m$ solutions are
\begin{equation}
\nu_j = \lambda + \epsilon \omega^j + E_j, \quad 0 \leq j \leq 2m -1, \quad \omega = \exp \left( \frac{2\pi}{2m} \right), \quad \abs{E_j} \leq 3 B \abs{\epsilon}^2.  
\end{equation}
In particular, 
\begin{equation} \label{eq:permissiblerange}
\frac{1}{2} \abs{\epsilon} \leq \abs{\nu_j - \lambda} \leq 2 \abs{\epsilon} < \rho \text{ for all }j, \, \,0 \leq j \leq 2m - 1.  
\end{equation}
\end{lem}
\begin{proof}
Keeping in mind \eqref{eq:localexpand}, and letting $z = \nu - \lambda$, let us analyze an equation
\begin{romansub}
\begin{gather}
c z^p (1 + g(z)) = w^p, \quad p \in \NN, c \neq 0\label{eq:pthpowerformula} \\
g(0) = 0, \, g \text{ analytic if } \abs{z} \leq R, \text{ and } \abs{g(z)} \leq \frac{1}{4} \text{ in this disc.} \label{eq:Rcond}\\
\end{gather} 
\end{romansub}
Let 
\begin{equation} 
B \declare \max \bracepair{g^{\prime}(z): z \leq R}. \label{eq:derbd}
\end{equation}
Put
\begin{equation} \label{eq:logformula}
h(z) = \log[1 + g(z)] = \sum_{n = 1}^{\infty} (-1)^{n + 1} \frac{[g(z)]^n}{n}.
\end{equation}
Then
\begin{romansub}
\begin{gather}
\abs{h(z) - g(z)} \leq \frac{2}{3} \abs{g(z)}^2 \leq \frac{1}{6} \abs{g(z)},\\
\frac{5}{6} \abs{g(z)} \leq \abs{h(z)} \leq \frac{7}{6} \abs{g(z)} < \frac{7}{6} \cdot \frac{1}{4} < \frac{1}{3}, \text{ and } \label{eq:sizecompare}\\
h^{\prime}(z) = \frac{g^{\prime}(z)}{1 + g(z)}, \textnormal{ so }\abs{h^{\prime}(z)} \leq \frac{4}{3} \abs{g^{\prime}(z)} \leq \frac{4}{3} B \text{ and } \abs{h(z)} \leq \frac{4}{3} B \abs{z}.  \label{eq:hsizebound}
\end{gather}
\end{romansub}
If $\omega^p = 1$, $\omega = \exp \left( \frac{2 \pi i}{p} \right)$, the equation \eqref{eq:pthpowerformula} splits into a series of $p$ equations
\begin{equation} \label{eq:pthroot}
c^{1/p} z \exp \left( \frac{1}{p} h(z) \right) = w \omega^j, \quad 0 \leq j < p,
\end{equation}
where $c^{1/p}$ is any $p$th root of $c$.  Each of them, for small enough $w$, has one and only one solution $z_j$, as it follows from analysis of the inverse function $z(\zeta)$ for 
\begin{equation} \label{eq:changecoords}
z \mapsto \zeta = \mathsf{X}(z) \equiv \zeta \exp \left( \frac{1}{p} h(z) \right), \quad \abs{z} \leq \frac{1}{4} \cdot \frac{1}{1 + B}; 
\end{equation}
it is well-defined for small $z$ because
\[
\mathsf{X}(0) = 0, \text{ and } \leiboneat{z}{\mathsf{X}(z)}{0} = 1.
\]
But we need good inequalities.  By \eqref{eq:changecoords} 
\begin{equation} \label{eq:invstart}
z = \Psi(z; \zeta) \declare \zeta e^{- \frac{1}{p} h(z)} = \zeta + \zeta \left[ e^{-\frac{1}{p} h(z)} - 1 \right],
\end{equation}
i.e., 
\begin{equation} \label{eq:fixpoint}
z = z(\zeta) \text{ is a fixed point of }\Psi(\cdot ; \zeta).
\end{equation}

\begin{bigclm}
If 
\begin{equation}\label{eq:sizerules}
\abs{\zeta} \leq \frac{1}{6} \cdot \frac{1}{1 + B}, 
\end{equation}
then 
\begin{equation}
\Psi: \mathbb{D}_{\kappa} \to \mathbb{D}_{\kappa} \text{ is contractive, for }\kappa = \frac{1}{4} \cdot \frac{1}{1 + B}.  
\end{equation}
\end{bigclm}
\begin{proof}
Indeed, $e^{1/3}< \frac{3}{2}$, so
\begin{equation} \label{eq:PsiboundsA}
\begin{split}
\abs{\Psi(z)} = \abs{ \zeta e ^{- \, \frac{1}{p} h(z)}} \leq \abs{\zeta} e^{\frac{1}{3}} < \frac{1}{6} \cdot \frac{1}{1 + B} \cdot \frac{3}{2} = \frac{1}{4} \cdot \frac{1}{1 + B},
\end{split} 
\end{equation}
and for $z, t \in \mathbb{D}_{\kappa}$, by \eqref{eq:hsizebound}, \eqref{eq:sizerules}
\begin{equation}
\begin{split} \label{eq:contract}
\abs{\Psi(z) - \Psi(t)} &\leq \abs{\zeta} \cdot \abs{e^{ - \frac{1}{p} h(z)} - e^{- \, \frac{1}{p} h(t)}}\\
& \leq \frac{1}{6} \cdot \frac{1}{1 + B} \cdot e^{1/3} \abs{h(z) - h(t)}\\
& \leq \frac{1}{4} \cdot \frac{1}{1 + B} \cdot \frac{4}{3} \cdot B \cdot \abs{ z - t} < \frac{1}{3} \abs{z - t}.
\end{split}
\end{equation}
\end{proof}
The Claim is proven.  It implies the existence of $z(\zeta) \in \eqref{eq:fixpoint}$ which gives solutions of \eqref{eq:pthroot} 
\begin{equation} \label{eq:pthsols}
z_j = z (\zeta_j) = \epsilon \omega^j, \quad \epsilon = w c^{-1/p}, \quad 0 \leq j < p.
\end{equation}
Additional properties of $z(\zeta)$ come from \eqref{eq:invstart} --- and \eqref{eq:PsiboundsA}, \eqref{eq:sizecompare}
\begin{equation} \label{eq:zetabounds}
\begin{split}
\abs{z - \zeta} & \abs{\zeta \left[ e^{ - \, \frac{1}{p} h(z) } - 1 \right]}\\
& \leq \abs{\zeta} e^{1/3} \abs{h(z)} \leq \abs{\zeta} \cdot \frac{3}{2} \cdot B \cdot \frac{4}{3} \cdot \abs{z}\\
& \leq \abs{z} 2 B \cdot \frac{3}{2} \abs{\zeta} \leq 3 B \abs{\zeta}^2.
\end{split}
\end{equation}
For solutions \eqref{eq:pthsols}, we have if 
\begin{equation} \label{eq:epsbd}
\abs{\epsilon} \leq \frac{1}{6 \abs{c}^{1/p} }\cdot \frac{1}{1 + B},
\end{equation}
\begin{equation} \label{eq:jthError}
z_j = \epsilon \omega^j + E_j, \quad \abs{E_j} \leq 3 B \abs{\epsilon}^2, \quad 0 \leq j < p.  
\end{equation} 

So far, \eqref{eq:permissiblerange} has not been explained, but if $\abs{\epsilon} \leq \min \bracepair{ \frac{1}{6B}, \frac{\rho}{3}}$, then $\abs{E_j} \leq 3B \left( \frac{1}{6B} \right)  \abs{\epsilon} = \frac{\abs{\epsilon}}{2}$, and hence
\[
\abs{z_j} \geq \abs{\epsilon \omega^j} - \abs{E_j} \geq \abs{\epsilon} \left( 1- \frac{1}{2} \right)
\]
and
\[
\abs{z_j} \leq \abs{\epsilon \omega^j} + \abs{E_j}  \leq \abs{\epsilon} \left( 1+ \frac{1}{2} \right)  \leq \frac{\rho}{3} \cdot \frac{3}{2} = \frac{\rho}{2} < \rho.  
\]
This explains that we may take $\delta = \sqrt{\abs{c} \cdot \min \bracepair{\frac{1}{(6B)^p}, \frac{1}{\abs{c} [6 (1 + B)]^p}, \left( \frac{\rho}{2} \right)^p}}$.
\end{proof}
\subsection{Case: real perturbation}  Suppose $z = r$ is real in \eqref{eq:pcnpertdef}, i.e.,
\[
\pcnpert{r}{b} (y(x)) = - y^{\prime \prime} + \left( \frac{x^2}{4} - \frac{1}{2} \right) y(x) + r (\pointmass{x - b} - \pointmass{x + b}), \, \, b > 0.
\]
\begin{lem} \label{lem:selfadj}
For $r \in \RR$, $\pcnpert{r}{b}$ is self-adjoint; consequently, $\spec{\pcnpert{r}{b}} \subseteq \RR$.  
\end{lem}
\begin{proof}
For real $r$, we now show that the quadratic form from which $\pcnpert{r}{b}$ is formed, i.e.,
\begin{equation}
\begin{split}
\ourformlong{r}{b}{u}{v} &\declare \ltwopair{u^{\prime}(x)}{v^{\prime}(x)}{\RR} + \frac{1}{4}\ltwopair{xu(x)}{xv(x)}{\RR} - \frac{1}{2}\ltwopair{u(x)}{v(x)}{\RR} \\
& \quad + r u(b) \overline{v(b)} - r u(-b) \overline{v(-b)}\\
\opdom{\ourform{r}{b}} & \declare \bracepair{f(x) \in \sobh{\RR}{1}: x f(x) \in L^2(\RR)}
\end{split}
\end{equation}
is \emph{semi-bounded below}, i.e.,
\[
\ourformlong{r}{b}{u}{u} \geq - c \ltwonorm{u(x)}{\RR}^2, \quad c = c(r) > 0.
\]  
Indeed,  for all $f \in \sobh{\RR}{1}$, for all $\epsilon > 0$, there exists $T = T(\epsilon)$ such that
\begin{equation} \label{eq:smallbd}
\sup_{x \in \RR} \abs{f(x)} \leq \epsilon \ltwonorm{f^{\prime}(x)}{\RR} + T(\epsilon) \ltwonorm{f(x)}{\RR}.
\end{equation}
For $r = 0$, 
\[
\begin{split}
\ourformlong{0}{b}{u}{u} &=  \ltwonorm{u^{\prime}(x)}{\RR}^2 + \frac{1}{4} \ltwonorm{x u(x)}{\RR}^2 - \frac{1}{2} \ltwonorm{u(x)}{\RR} \\
& \geq - \frac{1}{2} \ltwonorm{u(x)}{\RR}^2.
\end{split}
\]
For $r \neq 0$, we let $\epsilon = \frac{1}{2 \sqrt{\abs{r}}}$ and apply the inequalities \eqref{eq:smallbd} and $(\alpha + \beta)^2 \leq 2 (\alpha^2 + \beta^2)$.  Then
\[
\begin{split}
r u(b) \overline{u(b)} = r \abs{u(b)}^2 & > - \abs{r} \left(  \epsilon \ltwonorm{u^{\prime}(x)}{\RR} + T \ltwonorm{u(x)}{\RR} \right)^2\\
&> - \abs{r} \left( \frac{1}{2 \sqrt{\abs{r}}} \ltwonorm{u^{\prime}(x)}{\RR} + T\ltwonorm{u(x)}{\RR} \right)^2\\
& \geq - 2 \abs{r} \left( \frac{1}{4 \abs{r}} \ltwonorm{u^{\prime}(x)}{\RR}^2  + T^2 \ltwonorm{u(x)}{\RR}^2 \right)\\
& = - \frac{1}{2} \ltwonorm{u^{\prime}(x)}{\RR}^2 - 2 \abs{r} T^2 \ltwonorm{u(x)}{\RR}^2,
\end{split}
\]
and similarly for $-r u(-b) \overline{v(-b)}$.  Therefore,
\[
\begin{split}
\ourformlong{r}{b}{u}{u} & = \ltwonorm{u^{\prime}(x)}{\RR}^2 + \frac{1}{4} \ltwonorm{x u(x)}{\RR}^2 - \frac{1}{2} \ltwonorm{u(x)}{\RR}^2 + r \abs{u(b)}^2 - r \abs{u(-b)}^2\\
& > 0 - \left( 4 \abs{r} T^2 + \frac{1}{2} \right) \ltwonorm{u(x)}{\RR}^2.
\end{split}
\]
Hence, in all cases, $\ourform{r}{b}$ is semibounded below; as in \cite{MiSiArt}, one checks that it is closed.  Hence, the operator coming from the quadratic form is self-adjoint (e.g., \cite[Chapter VIII, Section 6, Theorem VIII.15, p. 279]{ReedSimon}).
\end{proof}
\begin{prop} \label{prop:realconsequences}  \mbox{}\\ \begin{enumerate}[label = (\roman*)] \item \label{enum:realroots} If $\lambda \in \CC$ and $\nuD{\lambda}= 0$, then $\lambda > 0$.
\item \label{enum:simpleroots} If $\lambda \in \RR \setminus \NN$ is a zero of $\mathsf{D}$, it is a simple zero.
\item \label{enum:poscoeff} If $\lambda \in \RR \setminus \NN$ is a zero of $\mathsf{D}$, then in the local series expansion
\[
\nuM{\nu} = c_2 (\nu - \lambda)^2 + \bigo{(\nu - \lambda)^3},
\]
$c_2 > 0$.\end{enumerate}
\end{prop}
\begin{proof}[Proof, Part~\ref{enum:realroots}]  Fix $\lambda$ such that $\nuD{\lambda} = 0$. If $\lambda \in \NN$, we know that $\lambda$ is real, and we know that $\nuD{0} = D_0(b) = e^{-b^2/4}$ is nonzero, so we reduce to the case where $\lambda \in \CC \setminus \Nz$.   Fix $\rho > 0$.  By Lemma~\ref{lem:approachlem}, for $r > 0$ large enough such that $\frac{1}{r} < \delta(\lambda, \rho)$, there exist several solutions of $\nuM{\nu} = \frac{1}{r^2}$ within $\rho$ of $\lambda$.  Yet these solutions cannot be nonnegative integers, for $\nuM{\nu}$ is not defined on the nonnegative integers.  Therefore, by Corollary~\ref{cor:Deminob}, these solutions are eigenvalues of $\pcnpert{r}{b}$, and by Lemma~\ref{lem:selfadj}, such eigenvalues must be real numbers.  Therefore, we have that the distance from $\lambda$ to the real line is bounded above by the distance from $\lambda$ to $\spec{\pcnpert{r}{b}}$, which is at most $\rho$ for sufficiently large $r$.  As $\rho > 0$ is arbitrary, we must have $\lambda \in \RR$.  

Moreover, by the integral representation from \eqref{eq:dnuintegrals} (e.g., \cite[Section 8.1.4, p. 328]{MOS}), if $\nu < 0$, then 
\[
\nuD{\nu} = \frac{ e^{ - b^2/4}}{\Gamma(-\nu)} \int_0^{\infty} t^{-\nu - 1} e^{-t^2/2  - bt} \, dt,
\]
which is a nonzero constant times the integral of a positive integrand, and hence is nonzero.  Therefore, the zeroes of $\nuD{\nu}$ avoid the negative real axis, and as already noted, $0$ is not a zero of $\nuD{\nu}$, so $\lambda > 0$.  
\end{proof}
\begin{proof}[Proof, Part~\ref{enum:simpleroots}]  Assume $m > 1$, $2m > 2$.  By Lemma~\ref{lem:approachlem}, see  \eqref{eq:maineqnflip}, even if $z$ is real, with
\begin{equation} \label{eq:powerseriesredo}
\nuM{\nu} = c_{2m} (\nu - \lambda )^{2m} (1 + g(\nu - \lambda )), \quad c_{2m} \neq 0,
\end{equation}
\begin{equation} \label{eq:epsilondef}
\epsilon^{2m} = \frac{1}{z^2 c_{2m}}, \text{ or }\epsilon = \left( \frac{1}{z^2} \right) ^{1/2m} c_{2m}^{-1/2m}, \quad \abs{z} \gg 1,
\end{equation}
($c_{2m}^{-1/2m}$ is \emph{any} $2m$th root of $c_{2m}^{-1}$), we have eigenvalues
\begin{equation} \label{eq:localrootexpansionredo}
\begin{split}
\nu_j &= \lambda + \epsilon \omega^j + E_j,  \quad \abs{E_j} \leq 3 B \abs{\epsilon}^2, \quad 0 \leq j < 2m\\
& = \lambda + \epsilon \omega^j ( 1 + \tau_j \epsilon), \quad \abs{\tau_j} \leq 3 B \abs{\epsilon}, \quad 0 \leq j < 2m
\end{split} 
\end{equation}
Therefore, 
\begin{equation}
\arg \left( \nu_j - \lambda \right) = \arg(\epsilon) +  \frac{2\pi}{2m} j + \sigma_j \quad \abs{\sigma_j} \leq 6 B \abs{\epsilon}, \quad 0 \leq j < 2m.
\end{equation}
and for $\epsilon$, $\abs{\epsilon} < \frac{1}{12B} \cdot \frac{\pi}{m}$, at most $2$ numbers in $\bracepair{\nu_j}_{j = 0}^{2m - 1}$, $\nu_j \in \eqref{eq:localrootexpansionredo} $, are real.  With $2m \geq 4 > 2$ we would get a non-real eigenvalue of the self-adjoint operator $\pcnpert{z}{b}$.  This contradiction implies that $m = 1$.  
\end{proof}
\begin{proof}[Proof, Part~\ref{enum:poscoeff}]  
Moreover, the same analysis shows that for $m = 1$ the coefficient $c_2 \neq 0$ should be positive: $c_2 > 0$.  Otherwise (i.e., with $c_2 < 0$), by \eqref{eq:epsilondef} with positive $z > 0$, $z \gg 1$, 
\begin{equation} \label{eq:epssquared}
\epsilon^2 = \frac{1}{z^2 c_2}, \text{ and } \epsilon = \pm \frac{1}{\sqrt{c_2}} \cdot \frac{1}{z}  \text{ is \emph{not} real.} 
\end{equation} 
 By \eqref{eq:localrootexpansionredo}
\begin{subequations} \label{eq:nupair}
\begin{align} 
\nu_0 & = \lambda + \epsilon + E_0 = \lambda + \epsilon (1 + s_0)\\
\nu_1 & = \lambda - \epsilon + E_1 = \lambda - \epsilon (1 + s_1),
\end{align}
\end{subequations}
where
\begin{equation} \label{eq:relerrorsize}
\abs{s_j} \leq \left( \frac{3B}{\abs{c_2}^{1/2}} \right) \cdot \frac{1}{\abs{z}}.
\end{equation}
For $z > 1$ large enough, both eigenvalues $\nu_0$, $\nu_1$ would not be real.  

\end{proof}
\subsection{Case: Pure Imaginary perturbation}
We now discuss $\pcnpert{ir}{b}$, $r \in \RR$, $b > 0$, as in \eqref{eq:pcnpertdef}.    
\begin{prop} \label{prop:nonrealLocal}
Fix $\lambda \in \RR_{> 0} \setminus \Nz$ such that $\nuD{\lambda} = 0$, and fix $\rho > 0$.  For $r > \frac{1}{\delta(\lambda, \rho)}$, there exists $2$ nonreal eigenvalues of $\pcnpert{ir}{b}$ in a radius-$\rho$ neighborhood of $\lambda$.
\end{prop}
\begin{proof}
We repeat the above analysis.  Let $z = ir$; then \eqref{eq:epssquared} becomes
\begin{equation}
\epsilon^2 = \frac{1}{z^2 c_{2}} = - \, \frac{1}{r^2 c_{2}}, \text{ and }   \epsilon = \pm \frac{1}{\sqrt{c_2}} \cdot \frac{i}{r} \text{ is \emph{not} real, by }\sqrt{c_2} \text{ real.}
\end{equation}  
Then \eqref{eq:nupair} and \eqref{eq:relerrorsize} still hold; hence, for $r$ large enough, the solutions of $\nuM{\nu} = \frac{1}{(ir)^2}$ near $\lambda$ are non-real.  Yet by Corollary~\ref{cor:Deminob}, these solutions are eigenvalues of $\pcnpert{ir}{b}$.  By Lemma~\ref{lem:approachlem}, for $r$ large enough, the non-real eigenvalues are within  $\rho$ of $\lambda$.    
\end{proof}

\section{Eigenvalue Localization around Integer Zeros of \texorpdfstring{$\mathsf{D}$}{D}}  \label{sec:nuInt}
\subsection{One Guaranteed Zero, and Several Nearby Zeros}   Although $D_0(x) = e^{-x^2/4}$ is never zero, it is certainly possible that $D_{n}(b) = 0$ for $n \in \NN$; indeed, $D_2(1) = 0$.  

By Corollary~\ref{cor:Deminob}, \eqref{eq:condrelabel},  $\nu \in \spec{\pcnpert{z}{b}}$ if and only if
\[
\frac{1}{\Gamma(- \nu)} - z^2 \nuD{\nu} \nuP{\nu} = 0.
\]
If $\nuD{n} = 0$, then $n \in \spec{\pcnpert{z}{b}} = 0$ for all $z$, for the left-hand-side of the above equation becomes 
\[
\frac{1}{\Gamma(-n)} - z^2 \nuD{n} \nuP{n} = 0 - z^2 \cdot 0 = 0.
\]
\begin{lem}
If $n \in \NN$ is a zero of $\mathsf{D}$, then $n \in \spec{\pcnpert{z}{b}}$ for all $z \in \CC$.
\end{lem}
We claim, however, that neighborhoods of such integer zeros of $\mathsf{D}$ work in much the same way as non-integer zeros of $\mathsf{D}$ as a starting point for finding non-real eigenvalues of $\spec{\pcnpert{ir}{b}}$.  
\begin{figure}
\centerline{%
\includegraphics[scale = 0.5]{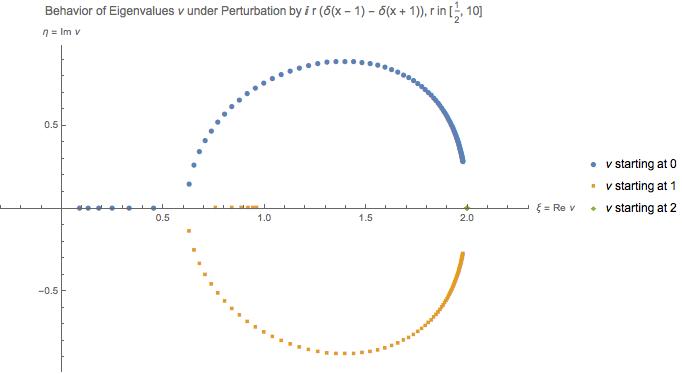}
} 
\caption{A graph of several eigenvalues of $\pcnpert{ir}{1}$ for $r \in \bracepair*{\frac{1}{2} + \frac{j}{10}: 0 \leq j \leq 95}$.  As $r$ increases, two of the eigenvalues are close to $n = 2$, but are non-real.}
\label{fig:intNuRoot} 
\end{figure}
Numerical evidence for this in the case $b = 1$ is given by Figure~\ref{fig:intNuRoot}.

The key observation is the second part of Lemma~\ref{lem:zerointerference}, revised here as
\begin{lem} \label{lem:zerocombineLoc}
If $n \in \NN$, and $\nuD{n} = 0$, then $\nuP{n}= 0$.  
\end{lem}
Thus, if $n \in \NN$ is a zero of $\mathsf{D}$ of order $m \geq 1$, then $n$ is a zero of $\Phi$ of order $q \geq 1$.  Then $[\nuD{\nu}]^2\nuP{\nu}$ has a zero of order at least $2m + q \geq 3$ at $\nu = n$, and thus $\nuM{\nu} = \Gamma(-\nu) [\nuD{\nu}]^2\nuP{\nu} $ (see \eqref{eq:maineqnelab}) has a removable singularity at $n$ and can be analytically continued there, with a zero of order  $2m + q - 1 \geq 2$ at $n$.  We call this extended function $\mathbf{M}(\nu)$ for clarity; write this function locally as
\begin{equation}
\mathbf{M}(\nu) = c_{2m + q - 1} (\nu - n)^{2m + q - 1}(1 + g(\nu - n)), \quad c_{2m + q - 1} \neq 0,
\end{equation}
where $g(0) = 0$ and $g$ is analytic in some disk; fix $R > 0$ such that $\abs{g(\xi)} \leq \frac{1}{4}$ whenever $\abs{\xi} \leq R$, and let $B$ the maximum of $g^{\prime}(\xi)$ on this interval.  We now establish the analogue of Lemma~\ref{lem:approachlem} for $\mathbf{M}(\nu)$. 
\begin{lem} \label{lem:localrootsInteger} 
Fix $\rho$, $0 < \rho < \frac{1}{2}$, and suppose that $\nuD{n} = 0$ for $n \in \NN$.  Then there exists a constant $\delta(n, \rho)$ such that if $\abs{z} > \frac{1}{\delta(n)}$, there exist $2m + q - 1$ solutions of $\nuM{\nu} = \frac{1}{z^2}$ (i.e., \eqref{eq:maineqnflip}), within the radius - $\rho$ - neighborhood of $\lambda$; more precisely, if $\epsilon \in \CC$ is defined such that $\epsilon^{2m + q - 1} = \frac{1}{z^2 c_{2m + q - 1}}$, then the $2m + q - 1$ solutions are
\begin{equation}
\begin{split}
\nu_j &= n + \epsilon \omega^j + E_j, \quad 0 \leq j < 2m + q - 1, \\
 \omega &= \exp \left( \frac{2\pi}{2m + q - 1} \right), \quad \abs{E_j} \leq 3 B \abs{\epsilon}^2.  
 \end{split}
\end{equation}
In particular, 
\begin{equation} \label{eq:permissiblerangeInteger}
\frac{1}{2} \abs{\epsilon} \leq \abs{\nu_j - n} \leq 2 \abs{\epsilon} < \frac{1}{2} \text{ for all }j, \, \,0 \leq j < 2m + q - 1.  
\end{equation}
\end{lem}
\begin{proof}
By repeating the proof of Lemma~\ref{lem:approachlem}, for sufficiently large $z$, there exist $2m + q - 1$ solutions of $\mathbf{M}(\nu) = \frac{1}{z^2}$ with the above properties.  However, by the first inequality of \eqref{eq:permissiblerangeInteger}, none of the solutions $\nu_j$ are equal to $n$ and by the latter two inequalities, $\abs{\nu_j - n} < \rho < \frac{1}{2} < 1$, no solution $\nu_j$ can equal any integer other than $n$.  Therefore, no $\nu_j$ is integer, which means that the solutions are solutions of the \emph{unextended} equation $\nuM{\nu} = \frac{1}{z^2}$ as well.  
\end{proof}

\subsection{Case: Real Perturbation} 
We now prove the analogue of Proposition~\ref{prop:realconsequences} for integer zeros of $\mathsf{D}$.  

\begin{prop} \label{prop:realconsequencesN}  
\begin{enumerate}[label = (\roman*)]
\item If $n \in \NN$ is a zero of $\mathsf{D}$ of order $m$, and it is a zero of $\Phi$ of order $q$, then both zeros are simple; i.e., $\mathbf{M}$ has a zero of order exactly $2$ there.  
\item If $n \in \NN$ is a zero of $\mathsf{D}$, then in the series expansion
\[
\mathbf{M}(\nu)= c_2(\nu - n)^2 + \bigo{(\nu - n)^3}, 
\]
$c_2 > 0$.  
\end{enumerate}
\end{prop} 
\begin{proof}[Proof, part (i)] If $n \in \NN$ is a zero of $\mathsf{D}$ of order $m$, and $n$ is a zero of $\Phi$ of order $q$, then the order of the zero of $\mathbf{M}$ at $n$ is $2m + q - 1$, since $[\nuD{\nu}]^2$ gives a zero of order $2m$ at $n$, and $\Gamma(- \nu)$ takes away an order of the zero.  Thus, let us write $\mathbf{M}$ locally as
\begin{equation} \label{eq:mextendexpand}
\mathbf{M}(\nu) = c_{2m + q - 1} (\nu - n)^{2m + q - 1}(1 + g(\nu - \lambda)), \quad c_{2m + q - 1} \neq 0.  
\end{equation}
Let
\begin{equation} \label{eq:epsilonIntLocal}
\epsilon^{2m + q - 1} = \frac{1}{z^2 c_{2m + q - 1}}, \text{ or } \epsilon = \left( \frac{1}{z^2} \right)^{1/(2m + q - 1)} c_{2m + q - 1}^{-1/(2m + q - 1)}, \abs{z} \gg 1,
\end{equation}
where $c_{2m + q - 1}^{-1/(2m + q - 1)}$ is any $2m  + q - 1$th root of $c_{2m + q - 1}^{-1}$.  By Lemma~\ref{lem:localrootsInteger},  letting $z$ be real, there exist $2m + q - 1$ solutions of $\mathbf{M}(\nu) = \frac{1}{z^2}$ in a radius-$\frac{1}{2}$-neighborhood of $n$, namely
\begin{equation} \label{eq:localsolInt}
\begin{split}
\nu_j &= n + \epsilon \omega^j + E_j, \abs{E_j} \leq 3 B \abs{\epsilon}^2, \quad 0 \leq j < 2m + q - 1,\\
& = n + \epsilon \omega^j (1 + \tau_j \epsilon), \quad \abs{\tau_j} \leq 3 B \abs{\epsilon}, 0 \leq j < 2m + q - 1,
\end{split}
\end{equation}  where $ \omega = \exp \left( \frac{2\pi}{2m + q - 1} \right)$.  By Corollary~\ref{cor:Deminob}, these solutions are eigenvalues of $\pcnpert{z}{b}$, and by Lemma~\ref{lem:selfadj}, they are real.

Therefore, 
\begin{equation}
\arg(\nu_j - n) = \frac{2\pi}{2m + q - 1} j +  \sigma_j, \, \l, \abs{\sigma_j} \leq 6 B \abs{\epsilon}, \, \, 0 \leq j < 2m + q - 1,
\end{equation}
and for $\epsilon$ such that $\abs{\epsilon} < \frac{1}{12B} \cdot \frac{2\pi}{2m + q - 1}$, at most $2$ numbers in $\bracepair{\nu_j}_{j = 0}^{2m + q - 2}$, $\nu_j \in \eqref{eq:localsolInt}$, are real.  With $2m + q - 1 \geq 3 > 2$, we would get a non-real eigenvalue of the self-adjoint operator $\pcnpert{z}{b}$.    This contradiction implies that $2m + q - 1 \leq 2$, and with $m \geq 1$, $q \geq 1$, the only resolution is $m = q = 1$; i.e., both roots are simple.  
\end{proof} 
\begin{proof}[Proof, part (ii)]  For $m = q = 1$, the coefficient $c_2 \neq 0$ should be positive: i.e., $c_2 > 0$.  Otherwise, by \eqref{eq:epsilonIntLocal}, with positive $z$, $z \gg 1$,
\begin{equation} \label{eq:epsnegInt}
\epsilon^2 = \frac{1}{z^2 c_2}, \, \text{ and } \epsilon = \pm \frac{1}{\sqrt{c_2}} \cdot \frac{1}{z} \text{ is \emph{not} real.}
\end{equation}
By \eqref{eq:localsolInt},
\begin{equation} \label{eq:twosplitInt}
\begin{split}
\nu_0 &= n + \epsilon + E_0 = \lambda + \epsilon (1 + s_0)\\
\nu_1 &= n - \epsilon + E_1 = \lambda - \epsilon (1 + s_1)\\
\end{split}
\end{equation}
where 
\begin{equation} \label{eq:errorsInt}
\abs{s_j} \leq \left( \frac{3B}{\abs{c_2}^{1/2}} \right) \cdot \frac{1}{\abs{z}} .
\end{equation}
For $z > 1$ large enough, both solutions $\nu_0$, $\nu_1$ would not be real, but they are real, again by Corollary~\ref{cor:Deminob} and Lemma~\ref{lem:selfadj}.  
\end{proof}

\subsection{Case: Imaginary Perturbation}
\begin{prop} \label{prop:nonrealLocalInt}
If $n \in \NN$ is such that $\nuD{n} = 0$, and $\rho > 0$, then for $r > \frac{1}{\delta(n, \rho)}$, there exists $2$ nonreal eigenvalues of $\pcnpert{ir}{b}$ (see \eqref{eq:pcnpertdef}) in the radius-$\rho$-neighborhood of $n$.   
\end{prop}
\begin{proof} We repeat the above analysis.  If $ z = ir$, then \eqref{eq:epsnegInt} becomes
\[
\epsilon^2 = \frac{1}{z^2 c_2} = - \, \frac{1}{r^2 c_2}, \text{ and } \epsilon = \pm \frac{1}{\sqrt{c_2}} \cdot \frac{i}{r} \text{ is \emph{not} real, by } \sqrt{c_2} \text{ real.}
\]
Then \eqref{eq:twosplitInt} and \eqref{eq:errorsInt} still hold; hence, for $r$ large enough, the solutions of $\nuM{\nu} = \frac{1}{(ir)^2}$ near $n$ are non-real.  yet by Corollary~\ref{cor:Deminob}, these solutions are eigenvalues of $\pcnpert{ir}{b}$.  
\end{proof}

\section{Proof of Theorem~\ref{thm:infinityHO}} \label{sec:Unify} Consider the family of operators $\pcnpert{re^{i \pi/2}}{b} = \pcnpert{ir}{b}$.  Consider the first $p$ roots $0 < \lambda_1 < \dotsc < \lambda_p$ of $\nuD{\nu}$.  By Lemma~\ref{lem:approachlem} and Proposition~\ref{prop:realconsequences} (or Lemma~\ref{lem:localrootsInteger} and Proposition~\ref{prop:realconsequencesN}), for each $q$, $1 \leq q \leq p$, the equation $\nuM{\nu} = \frac{1}{z^2}$ (or $\mathbf{M}(\nu) = \frac{1}{z^2}$) is of the form
\begin{equation} \label{eq:localexpandFinal}
c_2(\nu - \lambda_q)^2 [1 + g_q(\nu - \lambda_q)] = \frac{1}{(ir)^2} = - \frac{1}{r^2}, \quad c_2 > 0,
\end{equation}
where $g(z)  \leq C_3$ if $\abs{z} \leq \rho_q$, and we may enforce $\rho_q < \frac{1}{2} \min_{2 \leq q \leq p} \lambda_q - \lambda_{q - 1}$.  By Propositions~\ref{prop:nonrealLocal} and \ref{prop:nonrealLocalInt}, for each $q$, $1 \leq q \leq p$, there exists a constant $\delta(\lambda_q, \rho_q)$ such that for $r > \frac{1}{\delta(\lambda_q, \rho_q)}$, there exist $2$ non-real solutions of \eqref{eq:localexpandFinal} in the radius-$\rho_q$-neighborhood of $\lambda_q$ (call them $\kappa_j^{(q)}(r)$), $j = 0, 1$), and by Corollary~\ref{cor:Deminob}, these solutions are eigenvalues of $\spec{\pcnpert{ir}{b}}$.  Moreover, by $\rho_q < \frac{1}{2} \min_{2 \leq q \leq p} \lambda_q - \lambda_{q - 1}$, these neighborhoods are disjoint, so if $q \neq t$, $\kappa_j^{(q)}(r) \neq \kappa_k^{(t)}(r)$, $j, k$ in $\bracepair{0, 1}$.     Therefore, we see that for some $r^*$,
\begin{equation}
r \geq r^* \Rightarrow \card{\spec \pcnpert{ir}{b} \setminus \RR} \geq 2 p.
\end{equation}
This holds for all $p \in \NN$, so letting 
\[
\mathcal{N}_{\textnormal{PC}}(r) = \mathcal{N}_{\textnormal{PC}}(r; b) \declare \card{\spec \pcnpert{ir}{b} \setminus \RR},
\]
we have the following.
\begin{thm} \label{eq:countlarger}
$\lim_{r \to \infty} \mathcal{N}_{\textnormal{PC}}(r; b) = \infty$.
\end{thm}

By the change of variables in Section~\ref{subsect:changevar}, we establish Theorem~\ref{thm:infinityHO}, as well; i.e., $\lim_{r \to \infty} \mathcal{N}_{\textnormal{HO}}(r) = \infty$.

\section{Conclusion} \label{sec:Conclude} Two general questions should be clarified.

First, we know that
\begin{align}
\mathcal{N}(r) & \leq c (r \log (er))^2 & & \text{\cite[Thm. 4.4, (4.38), p. 4082]{MitPub2015},}\\
\intertext{and}
\lim_{r \to \infty} \mathcal{N}(r) &= \infty && \text{Theorem~\ref{thm:infinityHO}, \eqref{eq:TheTheorem}, p.~\pageref{thm:infinityHO} above.}
\end{align} 
But the gap between the estimates for $\mathcal{N}$ from above and below is too wide.  

Second, how to the eigenvalues $\lambda(ir)$ move, $0 \leq r < \infty$?  Sections~\ref{sec:nuNotInt}, \ref{sec:nuInt} tell a lot about $\lambda(ir)$, $r \gg 1$, close to zeros of $\mathsf{D}$.  But, for example, could some $\lambda_j(ir)$, $0 \leq r \ll 1$, go to $\infty$ when $r \to \infty$?  Numerics hint that it could not happen, but there is no formal (rigorous) argument to explain this phenomenon.
\nocite{BakMit2018}

\printbibliography

\end{document}